\documentclass[11pt,a4paper]{article}

\usepackage[utf8]{inputenc}
\usepackage[T1]{fontenc}
\usepackage{amsmath, amsthm, amssymb}
\usepackage{geometry}
\usepackage{hyperref}
\usepackage{enumitem}
\usepackage{xcolor}
\usepackage{algorithm}
\usepackage{algpseudocode}
\usepackage{graphicx}
\graphicspath{{figures/}}
\usepackage{subcaption}  

\newtheorem{theorem}{Theorem}[section]
\newtheorem{proposition}[theorem]{Proposition}
\newtheorem{lemma}[theorem]{Lemma}

\theoremstyle{definition}
\newtheorem{definition}[theorem]{Definition}
\theoremstyle{remark}
\newtheorem{remark}[theorem]{Remark}
\theoremstyle{example}

\title{Controlled Swarm Gradient Dynamics}
\author{
    Louison Aubert\thanks{E-mail: \href{mailto:louison.aubert@tse-fr.eu}{louison.aubert@tse-fr.eu}} \\
    \small Toulouse School of Economics, University of Toulouse Capitole, Toulouse, France
}
\date{\today}

\begin{document}

\maketitle

\begin{abstract}
\small
We consider the global optimization of a non-convex potential
\(U : \mathbb{R}^d \to \mathbb{R}\) and extend the controlled simulated annealing
framework introduced in~\cite{controlledSA} to the class of swarm gradient dynamics,
a family of Langevin-type mean-field diffusions whose noise intensity depends locally
on the marginal density of the process.
Building on the time-homogeneous model of~\cite{HuangMalik2025}, we first analyze its
invariant probability density and show that, as the inverse temperature parameter
tends to infinity, it converges weakly to a probability measure supported on the set
of global minimizers of \(U\).
This result justifies using this family of invariant measures as an annealing curve
in a controlled swarm setting.

Given an arbitrary non-decreasing cooling schedule, we then prove the existence of a
velocity field solving the continuity equation associated with the curve of invariant
densities.
Superimposing this field onto the swarm gradient dynamics yields a well-posed controlled
process whose marginal law follows exactly the prescribed annealing curve.
As a consequence, the controlled swarm dynamics converges toward global minimizers
with, in principle, arbitrarily fast convergence rates, entirely dictated by the
choice of the cooling schedule.
Finally, we discuss an algorithmic implementation of the controlled dynamics and
compare its performance with controlled simulated annealing, highlighting some
numerical limitations.
\end{abstract}

\paragraph{Acknowledgements.} 
The author would like to thank Stéphane Villeneuve for his guidance and constant support throughout this research. This work was supported by the grant ANR-17-EURE-0010 (Investissements d’Avenir program).

\section*{Introduction}

Global optimization of non-convex functions remains one of the most fundamental and challenging problems in modern optimization. When the objective function \( U : \mathbb{R}^d \to \mathbb{R} \) admits multiple local minima, classical gradient-based algorithms often fail to locate the global minimum. A popular probabilistic alternative is to recast the optimization task as a sampling problem, where the goal is to generate samples from a probability distribution \(\mu_\infty\) supported on the set of global minimizers of \(U\).

Among sampling-based methods, \emph{simulated annealing}, whose origins can be traced back to \cite{GemanHwang1986, ChiangHwangSheu1987}, stands out as one of the few approaches providing rigorous convergence guarantees.  
It is defined by the time-inhomogeneous Langevin diffusion
\begin{equation}
  \mathrm{d}X_t = -\nabla U(X_t)\,\mathrm{d}t 
  + \sqrt{\tfrac{2}{\beta(t)}}\,\mathrm{d}B_t,
\end{equation}
where \(\beta(t)\) is a non-decreasing function of time, called the \emph{cooling schedule}, and \((B_t)_{t\ge0}\) is a standard Brownian motion.  
Under suitable regularity assumptions on \(U\) and an appropriate choice of \(\beta(t)\), the law of \(X_t\) converges in distribution to the limiting probability measure \(\mu_\infty\) as \(t \to \infty\). Despite these theoretical guarantees, simulated annealing is notoriously slow in practice: its convergence rate is limited by metastability phenomena and cannot exceed a logarithmic speed.

This slow convergence can be understood through the invariant distribution of the Langevin diffusion at fixed inverse temperature, given by the Gibbs measure \(\mu_\beta \propto e^{-\beta U}\), \(\beta > 0\), which converges weakly to \(\mu_\infty\) as \(\beta \to \infty\)~\cite{Hwang1980}.  
In the simulated annealing regime, the process must remain sufficiently close to the evolving Gibbs curve \((\mu_{\beta(t)})_{t \ge 0}\) to guarantee convergence, which imposes severe restrictions on the admissible cooling schedules.

Hence, developing optimization or sampling schemes that retain theoretical guarantees while accelerating convergence remains a central challenge in global optimization.

\paragraph{Objective of the present work.}

A natural idea to accelerate annealing-type dynamics is to make the diffusion coefficient depend on the local density of the process itself. Indeed, the marginal law \(\rho^{X_t}\) implicitly reflects the geometry of the potential \(U\), as it tends to concentrate around local minima, precisely where additional stochasticity is needed to escape energy wells. This intuition motivates the introduction of \emph{Swarm Gradient Dynamics}, a class of McKean--Vlasov-type processes introduced in~\cite{BolteMicloVilleneuve2024}.

A different approach, introduced in~\cite{controlledSA}, consists in \emph{controlling} the simulated annealing dynamics so that its marginal law follows exactly the Gibbs curve \((\mu_{\beta(t)})_{t \ge 0}\). This is achieved by superposing a deterministic vector field onto the Langevin dynamics in order to correct the evolution of the marginal distribution. As a result, convergence is no longer constrained by metastability effects or functional inequalities, and arbitrarily fast cooling schedules become admissible.

In this paper, we propose to study a natural extension of this control strategy to swarm gradient dynamics. Our approach relies on the framework introduced in~\cite{HuangMalik2025}, which considers a particular class of time-homogeneous swarm gradient dynamics admitting an explicit invariant probability density. We show in Section~\ref{sec:limit} that this invariant measure converges weakly, as the inverse temperature parameter tends to infinity, to a probability measure supported on the set of global minimizers of \(U\). This property makes the corresponding swarm gradient dynamics a suitable candidate for extending the control methodology of~\cite{controlledSA}.

\paragraph{Swarm gradient dynamics.}

The \emph{Swarm Gradient Dynamics}, introduced in~\cite{BolteMicloVilleneuve2024}, are McKean--Vlasov-type stochastic processes defined by
\begin{equation}
  \mathrm{d}X_t
  = -\nabla U(X_t)\,\mathrm{d}t
  + \sqrt{\tfrac{2}{\beta(t)}\,\alpha(\rho^{X_t}(X_t))}\,\mathrm{d}B_t,
  \qquad
  X_0 \sim \rho_0,
  \label{eq:swarm_sde}
\end{equation}
where the function \(\alpha\) is derived from a convex function \(\varphi\) through
\[
  \alpha(r) = \frac{1}{r}\int_0^r s\,\varphi''(s)\,\mathrm{d}s,
  \qquad
  \varphi \in C^2((0,+\infty)),
  \quad
  \varphi''(r) > 0, \;
  \varphi'(0) = -\infty.
\]
The associated Fokker--Planck equation reads
\[
  \partial_t \rho_t
  = \nabla \cdot
  \big(
    \rho_t(\nabla U
    + \tfrac{1}{\beta(t)}\nabla \varphi'(\rho_t))
  \big).
\]

This structure combines gradient-based dynamics with local interactions depending on the particle density, hence the terminology \emph{swarm} dynamics.  
These processes are nonstandard McKean--Vlasov SDEs of Nemytskii type, since their coefficients depend locally—rather than globally—on the density. As a consequence, classical well-posedness results for McKean--Vlasov equations do not directly apply, and the analysis of such dynamics through SDE techniques becomes particularly delicate, especially in the time-inhomogeneous setting.

In recent years, optimal transport theory has provided powerful geometric tools to analyze stochastic processes. In particular, the space \(\mathcal{P}_2(\mathbb{R}^d)\) of probability measures with finite second moment, endowed with the 2-Wasserstein distance, allows one to interpret many diffusion equations as \emph{Wasserstein gradient flows} of suitable energy functionals~\cite{ambrosio2008gradient}.  
In this framework, the Fokker--Planck equation associated with the swarm gradient dynamics can be viewed as the time-dependent Wasserstein gradient flow of the functional
\[
  \mathcal{J}_{\beta(t)}(\rho)
  =
  \int_{\mathbb{R}^d}
  \Big(
    U(y)\rho(y)
    + \tfrac{1}{\beta(t)}\,\varphi(\rho(y))
  \Big)\,\mathrm{d}y.
\]
The extension of Wasserstein gradient flow theory to time-dependent functionals has been developed in~\cite{ferreira2015gradientflowstimedependentfunctionals}.

As in classical simulated annealing, where convergence is typically ensured through logarithmic Sobolev inequalities~\cite{HolleyKusuokaStroock1989}, this Wasserstein gradient flow perspective has motivated the study of functional inequalities guaranteeing convergence of swarm gradient dynamics.  In this context, the gradient flow formulation naturally leads to convergence results based on Łojasiewicz-type inequalities. Such an approach was developed in~\cite{BolteMicloVilleneuve2024}, where convergence was established for a restricted class of power-like function $\varphi$ and for processes evolving on one-dimensional compact Riemannian manifolds. Extending these results to higher-dimensional or non-compact settings remains an open problem.

A particular instance of swarm gradient dynamics was studied in~\cite{HuangMalik2025} in the time-homogeneous case. Using Wasserstein gradient flow techniques, the authors proved convergence toward an invariant probability measure admitting an explicit density formula in terms of the Lambert \(W\) function, for coercive potentials \(U\) on \(\mathbb{R}^d\), without compactness or dimensional restrictions.

\paragraph{Controlled stochastic annealing.}

The core idea of~\cite{controlledSA} is to control the simulated annealing dynamics so that its marginal law follows exactly the Gibbs curve. For each fixed time \(t\), the Gibbs measure \(\mu_{\beta(t)}\) satisfies
\[
  \mathcal{L}_t^* \mu_{\beta(t)} = 0,
\]
where \(\mathcal{L}_t\) denotes the infinitesimal generator of the Langevin diffusion with inverse temperature \(\beta(t)\). However, when \(\beta(t)\) varies in time, \(\mu_{\beta(t)}\) does not satisfy the full Fokker--Planck equation,
\[
  \partial_t \mu_{\beta(t)} \neq \mathcal{L}_t^* \mu_{\beta(t)}.
\]
To compensate for this mismatch, a deterministic velocity field \(v_t\) is introduced so as to satisfy the continuity equation
\[
  \partial_t \mu_{\beta(t)} + \nabla \cdot (v_t \mu_{\beta(t)}) = 0.
\]
Superposing this field onto the Langevin dynamics yields a controlled stochastic process whose marginal law exactly follows the prescribed Gibbs curve. Consequently, the convergence rate depends solely on the chosen cooling schedule.

The identification in~\cite{HuangMalik2025} of an invariant probability measure admitting an explicit density \(\rho_\beta\) depending on the inverse temperature parameter \(\beta\) makes it possible to extend this control strategy to swarm gradient dynamics. More precisely, it allows one to search for a time-dependent vector field \(v_t\) satisfying the continuity equation associated with a curve of probability densities \((\rho_{\beta(t)})_{t \ge 0}\), thereby controlling the time-inhomogeneous swarm gradient dynamics. As a result, the controlled process is guaranteed to converge independently of the choice of cooling schedule. Moreover, since the marginal law is explicitly prescribed, the dynamics is no longer nonlinear in its law, and the need to estimate the density is replaced by the estimation of the velocity field \(v_t\).

\paragraph{Related Work}

Stochastic methods for global optimization are classically rooted in Markov chain Monte Carlo techniques, most notably the Metropolis--Hastings algorithm for sampling Gibbs measures~\cite{RobertCasella2004}.
Simulated annealing extends this framework to optimization by introducing a time-dependent temperature schedule: at each iteration, a candidate state is proposed and accepted according to a Gibbs-type probability that decreases over time~\cite{GemanGeman1984}.
Under a logarithmic cooling schedule, classical simulated annealing converges to the set of global minimizers~\cite{HolleyStroock1988}.
Several variants have been proposed to accelerate convergence.
Fast simulated annealing replaces Gaussian proposal kernels by heavy-tailed Cauchy distributions, enabling larger jumps and allowing cooling schedules of order $1/t$~\cite{SzuHartley1987}.
Generalized simulated annealing further unifies these approaches by parametrizing both visiting and acceptance distributions~\cite{TsallisStariolo1996}.
Very fast simulated annealing proposes exponential cooling schedules, although convergence guarantees remain mostly heuristic~\cite{Ingber1989}.

A continuous-time perspective interprets simulated annealing as a Langevin diffusion with time-dependent temperature~\cite{GemanHwang1986,ChiangHwangSheu1987}.
As in the discrete-time setting, convergence in probability toward global minimizers is ensured under logarithmic cooling schedules~\cite{HolleyKusuokaStroock1989,Miclo1992}.
Several extensions modify the underlying dynamics in order to enhance exploration or facilitate numerical implementation.
Piecewise deterministic simulated annealing replaces the diffusion process by a piecewise deterministic Markov process, preserving logarithmic cooling requirements while reducing simulation cost~\cite{Monmarche2014}.
With the same cooling schedule constraint, kinetic annealing~\cite{JournelMonmarche2021} introduces underdamped Langevin dynamics by augmenting the state space with velocity variables, leading to improved performance at fixed temperature.
Another line of work replaces Brownian motion with Lévy processes, yielding Lévy-flight-based simulated annealing schemes that introduce rare but large jumps and allow polynomial cooling schedules~\cite{Pavlyukevich2007a,Pavlyukevich2007b,Pavlyukevich2008}.
However, these processes can be difficult to tune in practice so as to ensure convergence exclusively toward global minimizers rather than local ones.

A modern viewpoint on Langevin dynamics interprets the Fokker--Planck equation as the Wasserstein gradient flow of the Kullback--Leibler divergence~\cite{JordanKinderlehrerOtto1998,ambrosio2008gradient}.
This formulation naturally extends to time-dependent functionals~\cite{ferreira2015gradientflowstimedependentfunctionals}, allowing simulated annealing to be seen as a time-dependent gradient flow.
The efficiency of stochastic optimization methods is closely related to their ability to escape basins of attraction of local minima, a question that also arises in multimodal sampling.
In this context, by studying the gradient flow of the Kullback--Leibler divergence with respect to the Wasserstein--Fisher--Rao metric, \cite{LuLuNolen2019} introduced Langevin dynamics with birth--death mechanisms, yielding Gibbs samplers with exponential convergence rates independent of energy barriers.

As mentioned above, the Wasserstein perspective has also enabled the analysis of swarm gradient methods~\cite{BolteMicloVilleneuve2024,HuangMalik2025}.
Other noise-modulation strategies include so-called fraudulent algorithms, which scale the diffusion coefficient by the value of the potential itself.
Under knowledge of the global minimum value, such processes, on compact Riemannian manifolds of dimension at least two, converge almost surely to global minimizers~\cite{Miclo2025,BenaimMiclo2024}.

As previously mentioned, optimal transport theory has also motivated control-based approaches to simulated annealing.
In particular, \cite{controlledSA} proposed to control Langevin dynamics through the superposition of an external vector field, yielding arbitrarily fast convergence rates in theory.
An alternative control-based approach formulates the problem as a stochastic control problem on a state-dependent temperature~\cite{GaoXuZhou2022}.

Beyond density-dependent noise or control via vector-field superposition, our approach shares structural similarities with particle-based global optimization methods combining local exploration and global interaction.
A notable example is Consensus-Based Optimization (CBO), where particles interact through a weighted consensus point biased toward low-cost regions~\cite{CarrilloChoiTotzeckTse2018,KaliseSharmaTretyakov2023}; see~\cite{Totzeck2022} for a recent overview. Another example is Particle Swarm Optimization (PSO)~\cite{KennedyEberhart1995}, where a population of particles evolves while being attracted both to their own best positions found so far and to the best position found so far by the swarm. See~\cite{GrassiPareschi2021} for a study of continuous-time PSO and its mean-field connection to CBO.

Finally, several hybrid methods aim not at concentrating on the global minimizers, but rather at visiting both local and global minima.
Intermittent diffusion alternates between deterministic gradient descent phases and stochastic Langevin phases, allowing the process to reach a neighborhood of a global minimizer with arbitrarily high probability in finite time~\cite{ChowYangZhou2013}.
Switched diffusion processes combine Langevin exploration with dynamics specifically designed for saddle-point search, providing an efficient framework for escaping basins of attraction of local minima and traversing complex energy landscapes~\cite{JournelMonmarche2023}.

\paragraph{Organization of the paper}

The paper is organized as follows.
Section~\ref{sec:preliminaries} introduces the notation and recalls several results from the literature that will be used throughout the paper.
In Section~\ref{sec:homogeneous}, we present the time-homogeneous swarm gradient dynamics and explain how the existence of an explicit invariant probability density allows one to envisage a control of the process.
Section~\ref{sec:limit} is devoted to the study of the weak limit of this invariant density and to its convergence toward a measure supported on the set of global minima of the potential~$U$.
In Section~\ref{sec:continuity}, we prove the existence of a vector field satisfying the continuity equation associated with the curve $(\rho_t)_{t\in[0,T]}$ of invariant measures of the homogeneous swarm gradient dynamics, as well as the absolute continuity of $(\rho_t)_{t\in[0,T]}$.
The existence and well-posedness of the controlled process are then established in Section~\ref{sec:controlled}.
Section~\ref{sec:implementation} discusses the algorithmic implementation of the controlled dynamics.
Finally, Section~\ref{sec:numerical} presents several numerical experiments on non-convex optimization problems, together with a comparison with classical controlled simulated annealing.

\section{Mathematical Background}
\label{sec:preliminaries}

\paragraph{Notation}
We denote by \( \mathcal{P}(\mathbb{R}^d) \) the set of all probability measures on \( \mathbb{R}^d \), and by \( \mathcal{P}_2(\mathbb{R}^d) \) the subset of measures with finite second moment.  
Similarly, we let \( \mathcal{D}_2(\mathbb{R}^d) \) denote the set of probability \emph{densities} on \( \mathbb{R}^d \) with finite second moment.

For a measure \( \mu \in \mathcal{P}(\mathbb{R}^d) \), we write \( \mu(\mathrm{d}x) = f(x)\,\mathrm{d}x \) when \( \mu \) admits a density \( f \) with respect to the Lebesgue measure, where \( \mathrm{d}x \) denotes the Lebesgue measure on \( \mathbb{R}^d \).  

We denote by \( |\cdot| \) the Euclidean norm, \( \langle \cdot, \cdot \rangle \) the associated scalar product, \( \|\cdot\|_{\mathrm{Fr}} \) the Frobenius norm, and \( \langle \cdot, \cdot \rangle_{\mathrm{Fr}} \) its associated inner product.

Given two measures \( \mu \) and \( \nu \) defined on two metric spaces \( X \) and \( Y \), we denote by \( \Pi(\mu,\nu) \) the set of \emph{couplings} between \( \mu \) and \( \nu \), i.e., the set of all probability measures on \( X \times Y \) whose first and second marginals are \( \mu \) and \( \nu \), respectively.

\paragraph{Optimal transport}
  
The \emph{2-Wasserstein distance} between two measures \( \mu, \nu \in \mathcal{P}_2(\mathbb{R}^d) \) is defined by
\[
W_2^2(\mu,\nu)
:= \inf_{\pi \in \Pi(\mu,\nu)} 
\int_{\mathbb{R}^d \times \mathbb{R}^d} |x - y|^2 \, \mathrm{d}\pi(x,y).
\]
The space \( (\mathcal{P}_2(\mathbb{R}^d), W_2) \) is a complete metric space.  
The Wasserstein distance corresponds to the square root of the optimal value of the Kantorovich problem for the quadratic cost.  
According to Brenier’s theorem~\cite[Theorem 2.12]{villani2003topics}, if \( \mu \) is absolutely continuous with respect to the Lebesgue measure, then there exists a unique optimal coupling \( \pi^* \) of the form \( \pi^* = (\mathrm{id} \times T)_\# \mu \), where \( T \) (called the \emph{Monge map}) is the gradient of a convex function.

\paragraph{Absolutely continuous curves in Wasserstein space}
We now recall a few results that will be used in the proof of Theorem~\ref{thm:existence vt}.  
Since \( (\mathcal{P}_2(\mathbb{R}^d), W_2) \) is a metric space, we can define absolutely continuous curves of measures.

\begin{definition}[Absolutely continuous curves and metric derivative]
Let \( I \subset \mathbb{R} \) be an open interval.  
A curve \( \eta : I \to \mathcal{P}_2(\mathbb{R}^d) \), \( t \mapsto \mu_t \), is said to be \emph{absolutely continuous} if its metric derivative,
\[
  |\mu'|(t) := \lim_{h \to 0} \frac{W_2(\mu_t, \mu_{t+h})}{|h|},
\]
exists for Lebesgue almost every \(t \in I\) and belongs to \( L^1(I) \).
\end{definition}

This notion of absolute continuity is closely related to the continuity equation and to the existence of an associated velocity field.

\begin{theorem}[Absolutely continuous curves and the continuity equation {\cite[Theorem 8.3.1]{ambrosio2008gradient}}]
\label{thm:AGS_continuity}
Let \( I \subset \mathbb{R} \) be an open interval, and let \( t \mapsto \mu_t \in \mathcal{P}_2(\mathbb{R}^d) \) be an absolutely continuous curve.  
Then there exists a time-dependent vector field \( v_t(x) : I \times \mathbb{R}^d \to \mathbb{R}^d \) such that the pair \( (\mu_t, v_t) \) satisfies the continuity equation
\begin{equation}
\label{eq:continuity_equation}
  \partial_t \mu_t + \nabla \!\cdot\! (v_t \mu_t) = 0.
\end{equation}
Moreover, it holds that \( \|v_t\|_{L^2(\mu_t, \mathbb{R}^d)} \in L^1(I) \).

Conversely, if such a field \( v_t \) exists and satisfies~\eqref{eq:continuity_equation} together with \( \|v_t\|_{L^2(\mu_t, \mathbb{R}^d)} \in L^1(I) \), then the curve \( t \mapsto \mu_t \) is absolutely continuous in \( (\mathcal{P}_2, W_2) \).
\end{theorem}

When each measure \( \mu_t \) admits a density with respect to the Lebesgue measure, Brenier’s theorem yields an explicit form of the velocity field.

\begin{proposition}[Transport map characterization of the velocity field {\cite[Proposition 8.4.6]{ambrosio2008gradient}}]
\label{prop:charac_vt}
Let \( t \mapsto \mu_t \in \mathcal{P}_2(\mathbb{R}^d) \) be an absolutely continuous curve, and assume that each \( \mu_t \) is absolutely continuous with respect to the Lebesgue measure.  
Then the corresponding velocity field is given by
\[
  v_t = \lim_{h \to 0} \frac{T_{\mu_t \to \mu_{t+h}} - \mathrm{id}}{h}
  \quad \text{in } L^2(\mu_t),
\]
where \( T_{\mu_t \to \mu_{t+h}} \) is the unique optimal transport map from \( \mu_t \) to \( \mu_{t+h} \).
\end{proposition}

\begin{theorem}[Characterization of the minimal velocity field {\protect\cite[from Proposition 8.4.5 and Theorem 8.3.1]{ambrosio2008gradient}}]
\label{thm:minimal_velocity_field}
Let \( t \mapsto \mu_t \in \mathcal{P}_2(\mathbb{R}^d) \) be an absolutely continuous curve, and let \( v_t \in L^2(\mu_t; \mathbb{R}^d) \) satisfy~\eqref{eq:continuity_equation}.  
Then \( v_t \) is the unique vector field minimizing the \( L^2(\mu_t) \)-norm among all vector fields satisfying~\eqref{eq:continuity_equation} if and only if
\[
  v_t \in \overline{\{ \nabla \varphi : \varphi \in C_c^\infty(\mathbb{R}^d) \}}^{L^2(\mu_t)}
  = \mathrm{Tan}_{\mu_t} \mathcal{P}_2(\mathbb{R}^d),
  \quad \text{for a.e. } t \in I.
\]
\end{theorem}

\paragraph{Sobolev spaces}
We now recall several definitions and properties of Sobolev spaces that will be used in the proof of Theorem~\ref{thm:existence vt}.

\begin{definition}[\( L^2 \)-Poincaré constant]
Let \( \mu \in \mathcal{P}(\mathbb{R}^d) \).  
The \( L^2 \)-Poincaré constant of \( \mu \) is defined by
\[
C_P(\mu) := \inf \left\{ C \geq 0 : 
\|f\|_{L^2(\mu)}^2 \leq C \|\nabla f\|_{L^2(\mu)}^2, 
\quad \forall f \in C_c^\infty(\mathbb{R}^d) \right\}.
\]
\end{definition}

Assume that \( C_P(\mu) < \infty \).  
We then define the Sobolev space \( H^1(\mu) \) associated with \( \mu \) as
\[
H^1(\mu) := \{ f \in L^2(\mu; \mathbb{R}) : \nabla f \in L^2(\mu; \mathbb{R}^d) \},
\]
where \( \nabla f \) denotes the weak gradient of \( f \).  
The space \( H^1(\mu) \) becomes a Hilbert space under the scalar product
\[
\langle f, g \rangle_\mu := \int f g\, \mathrm{d}\mu + \int \langle \nabla f, \nabla g \rangle\, \mathrm{d}\mu.
\]
By the Meyers–Serrin theorem, \( H^1(\mu) \) is the completion of \( C_c^\infty(\mathbb{R}^d) \) with respect to the norm induced by this scalar product (see Chapter~7 in~\cite{lieb2001analysis}).

We also use the semi-norm \( \|\nabla f\|_{L^2(\mu)} \) induced by
\[
\langle f, g \rangle_{\nabla,\mu} := \int \langle \nabla f, \nabla g \rangle\, \mathrm{d}\mu.
\]
Thanks to the Poincaré inequality, this semi-norm is equivalent to the full \( H^1(\mu) \)-norm (see Chapter~IX in~\cite{brezis}).  
Finally, we define the subspace of zero-mean functions:
\[
\dot{H}^1(\mu) := \{ f \in H^1(\mu) : \int f\, \mathrm{d}\mu = 0 \}.
\]
When \( C_P(\mu) < \infty \), the inner product \( \langle \cdot, \cdot \rangle_{\nabla,\mu} \) turns \( \dot{H}^1(\mu) \) into a Hilbert space.

\section{ From the Homogeneous Swarm Gradient Dynamics to the Controlled Annealing Strategy}
\label{sec:homogeneous}

In~\cite{HuangMalik2025}, the authors introduce a mean-field Langevin diffusion with a density-dependent temperature, which, from a structural viewpoint, can be regarded as a time-homogeneous version of the swarm gradient dynamics.
In this section, we first introduce this homogeneous model, then recall its main properties and invariant measure, and finally explain the controlled annealing strategy studied in the remainder of the paper.

\medskip

Let $U : \mathbb{R}^d \to \mathbb{R}$ be a $\mathcal{C}^1$ function, assumed to be nonnegative and possibly nonconvex. For $m > 1$ and an inverse temperature parameter $\beta \ge 1$, define the function $\varphi : [0, +\infty) \to \mathbb{R}$ by
\[
  \varphi(r) = r \ln r + \frac{r^m}{m-1},
\]
so that
\[
  \varphi'(r) = 1 + \ln r + \frac{m}{m-1}r^{m-1},
  \qquad 
  \varphi''(r) = \frac{1}{r} + m r^{m-2}.
\]
Hence, $\varphi$ is convex, of class $\mathcal{C}^2$ on $(0, +\infty)$, with $\varphi''(r) > 0$ and $\varphi'(0) = -\infty$.  

We also introduce
\[
  \alpha(r) = \frac{1}{r}\int_0^r s \varphi''(s)\,ds = 1 + r^{m-1}.
\]

The (time-homogeneous) Swarm Gradient process is then defined by the stochastic differential equation (SDE)
\begin{equation}
  dX_t = -\nabla U(X_t)\,dt 
  + \sqrt{\frac{2}{\beta}\,\alpha(\rho^{X_t}(X_t))}\,dB_t, 
  \qquad X_0 \sim \rho_0,
  \label{eq:sde}
\end{equation}
where $(B_t)_{t \ge 0}$ is a $d$-dimensional Brownian motion and $\rho^{X_t}$ denotes the marginal law of $X_t$.

The associated Fokker–Planck equation reads, for $\mu_t = \rho^{X_t}$,
\begin{equation}
  \partial_t \mu_t 
  = \nabla \cdot 
  \big(
    \mu_t( \nabla U 
    + \tfrac{1}{\beta}\nabla \varphi'(\mu_t))
  \big)
  = \nabla \cdot 
  \Big(
    \mu_t \nabla U 
    + \tfrac{1}{\beta}(\nabla \mu_t + \nabla \mu_t^m)
  \Big).
  \label{eq:fokker-planck}
\end{equation}

\medskip

Under the growth assumption
\begin{equation}
  \exists\, K > 0 \text{ such that } |\nabla U(y)| \leq K(1 + |y|), 
  \quad \forall y \in \mathbb{R}^d,
  \label{eq:gradV-growth}
\end{equation}
the process can be viewed as the Wasserstein subgradient flow of the functional
\[
  \mathcal{J}_\beta : \mathcal{P}_2(\mathbb{R}^d) 
  \to \mathbb{R} \cup \{+\infty\},
\]
defined for any $\mu \in \mathcal{P}_2(\mathbb{R}^d)$ by
\[
  \mathcal{J}_\beta(\mu) :=
  \begin{cases}
    J_\beta(\rho), & \text{if } \mu(dx) = \rho(x)\,dx , \\[0.5em]
    +\infty, & \text{otherwise},
  \end{cases}
\]
where
\[
  J_\beta(\rho)
  = \int_{\mathbb{R}^d}
  \Big(
    U(y)\rho(y)
    + \tfrac{1}{\beta}\varphi(\rho(y))
  \Big)\,dy,
\]
see Corollary~3.1 in~\cite{HuangMalik2025}.

Moreover, for any $\rho_0 \in D_2(\mathbb{R}^d)$  
such that $J_\beta(\rho_0) < +\infty$, there exists  
$\rho : [0, +\infty) \to D_2(\mathbb{R}^d)$ solving~\eqref{eq:fokker-planck},  
such that for $\mathcal{L}^1$-almost every $t \ge 0$,
\[
  \partial \mathcal{J}_\beta(\mu_t)
  =
  \Big\{
    \nabla U
    + \tfrac{1}{\beta}
      \big(
        \nabla \mu_t / \mu_t
        + \nabla \mu_t^m / \mu_t
      \big)
  \Big\},
\]
where $\mu_t(dx) = \rho_t(x)\,dx$, and $\partial \mathcal{J}_\beta(\mu)$ denotes the subdifferential of $\mathcal{J}_\beta$ at $\mu$ [Theorem~4.1, \cite{HuangMalik2025}].

\medskip

Furthermore, under the additional assumption
\begin{equation}
  \exists\, R, \delta > 0 \text{ such that } 
  U(y) \geq \Big(\tfrac{1}{\beta} (d + 2) + \delta\Big)\ln|y|, 
  \quad \forall\, |y| \geq R,
  \label{eq:V-lowerbound}
\end{equation}
there exists a unique minimizer of 
\[
  J_\beta : D_2(\mathbb{R}^d) \to \mathbb{R},
\]
given by
\[
  \rho_\beta(y)
  =
  \left(
    \frac{1}{m}
    W_0\!\left(
      m e^m
      e^{-(m-1)\beta (U(y) - C)}
    \right)
  \right)^{\!\frac{1}{m-1}},
\]
where $C$ is the unique positive constant ensuring that $\rho_\beta$ is a probability density [Theorem~6.1, \cite{HuangMalik2025}].  
Here, $W_0$ denotes the \emph{principal branch of the Lambert function}, defined by
\[
  \forall\, x, y > 0, \quad 
  x = y e^y \quad \Longleftrightarrow \quad y = W_0(x).
\]

Thus, $\rho_\beta(x)\,dx$ is the stationary distribution of the process~\eqref{eq:sde}.  
Finally, if $\nabla U$ is Lipschitz, Theorem~7.1 in~\cite{HuangMalik2025} ensures the existence of a sequence $(t_n)_{n \in \mathbb{N}} \subset [0, +\infty)$ such that $t_n \to +\infty$ and
\[
  \rho(t_n) \rightharpoonup \rho_\beta
  \quad \text{weakly in } \mathcal{P}(\mathbb{R}^d)
  \text{ as } n \to +\infty.
\]

\medskip

The process defined above can be interpreted as a \emph{modification of the Langevin diffusion} in which the temperature depends on the probability density.  
This dependence enhances exploration: in regions where trajectories tend to accumulate near local minima of $U$, the marginal densities $\rho^{X_t}$ take larger values, which increases the amplitude of the noise compared to the classical Langevin diffusion.  
Conversely, when a trajectory lies far from potential wells, the noise remains comparable to that of the standard Langevin model.  

To prevent \emph{back-and-forth} phenomena—i.e., the immediate reattraction of particles recently ejected from a local minimum—a uniform noise component, independent of the density, is incorporated into the dynamics.  
As a result, a particle that has escaped a local potential well will not be instantly drawn back into it.

\medskip

Since this process directly stems from the Langevin diffusion, it is natural to expect that, as the parameter $\beta$ increases, the stationary measure $\rho_\beta$ concentrates around the \emph{global minima} of $U$.  
Moreover, by making the process \emph{time-inhomogeneous}—that is, by replacing the constant parameter $\beta$ with a slowly varying cooling schedule
\[
  \beta : [0, +\infty) \to [1, +\infty),
\]
of class $\mathcal{C}^1$, nondecreasing and diverging to $+\infty$ sufficiently slowly—it becomes possible to obtain a stationary law supported on the global minima of $U$.

\medskip

In what follows, we therefore consider a \emph{cooling schedule} 
$\beta : [0, +\infty) \to [1, +\infty)$, $\mathcal{C}^1$ and nondecreasing,  
and we study the following process:
\begin{equation}
  dX_t = -\nabla U(X_t)\,dt 
  + \sqrt{\frac{2}{\beta(t)}\,\alpha(\rho^{X_t}(X_t))}\,dB_t, 
  \qquad X_0 \sim \rho_0,
  \label{eq:sde-timeinhom}
\end{equation}
where $(B_t)_{t \ge 0}$ is a $d$-dimensional Brownian motion and $\rho^{X_t}$ denotes the marginal law of $X_t$.

The associated Fokker–Planck equation is then
\begin{equation}
  \partial_t \mu_t 
  = \nabla \cdot 
  \big(
    \mu_t\big( \nabla U 
    + \tfrac{1}{\beta(t)}\nabla \varphi'(\mu_t)
  \big)\big)
  = \nabla \cdot 
  \Big(
    \mu_t \nabla U 
    + \tfrac{1}{\beta(t)}(\nabla \mu_t + \nabla \mu_t^m)
  \Big).
  \label{eq:fokker-planck-timeinhom}
\end{equation}

\medskip

The aim of this work is to apply the control-based framework developed in~\cite{controlledSA} to this process.  
Specifically, we aim to construct a vector field $(v_t)_{t \ge 0}$ satisfying the continuity equation
\[
  \partial_t \rho_t - \nabla \cdot (\rho_t v_t) = 0,
\]
where $\rho_t := \rho_{\beta(t)}$,  
such that, when superimposed on the process~\eqref{eq:sde-timeinhom},  
the marginals exactly follow the trajectory $(\rho_t)_{t \ge 0}$.  
Hence, the convergence rate depends solely on the choice of the cooling schedule $\beta(t)$.

For simplicity, we will use the same notation $\rho_t$ to denote both the probability measure and its Lebesgue density, i.e. $\rho_t(dx) = \rho_t(x)\,dx$.

Establishing the existence of a field $(v_t)_{t \geq 0}$ satisfying the above properties allows us to use the contrapositive of Theorem~\ref{thm:AGS_continuity}, and thereby to prove the absolute continuity of the curve $(\rho_t)_{t \ge 0}$.  
This property is crucial, as it enables the use of the optimal-transport characterization given by Proposition~\ref{prop:charac_vt}, which is extremely useful for approximating $v_t$ and discretizing the process under study.

\medskip

In the sequel, variants of assumption~\eqref{eq:V-lowerbound} will be used depending on the context.  
We now state a simple proposition illustrating its impact on the regularity of the measures~$\rho_t$.

\begin{proposition}\label{prop:moment}
Let $\ell \in \mathbb{N}$.  
Assume that there exist constants $R, \delta > 0$ such that
\begin{equation}
  U(y) \geq 
  \Big(\tfrac{1}{\beta(0)}(d + \ell) + \delta\Big)\ln|y|,
  \quad \forall\, |y| \ge R.
  \label{eq:l-lowerbound}
\end{equation}
Then, for all $t \ge 0$, we have $\rho_t \in \mathcal{D}_\ell(\mathbb{R}^d)$.
\end{proposition}

\begin{proof}
Fix $t \ge 0$.  
By definition of $\rho_t$, there exists a constant $C(t) > 0$ such that
\[
  \rho_t(x)
  =
  \left(
    \frac{1}{m}W_0\!\left(
      m e^m e^{-(m-1)\beta(t)(U(x) - C(t))}
    \right)
  \right)^{\!\frac{1}{m-1}},
\]
where $W_0$ denotes the principal branch of the Lambert function.  
Since $W_0(z) \le \ln(1+z) \le z$ for all $z > 0$, we obtain
\[
  \rho_t(x) 
  \le 
  \left(
    e^{m} e^{-(m-1)\beta(t)(U(x) - C(t))}
  \right)^{\!\frac{1}{m-1}}
  = e^{\frac{m}{m-1} + \beta(t)C(t)} e^{-\beta(t)U(x)}.
\]
Hence,
\[
  \int_{\mathbb{R}^d} |x|^{\ell}\rho_t(x)\,dx
  \le 
  e^{\frac{m}{m-1} + \beta(t)C(t)}
  \int_{\mathbb{R}^d} |x|^{\ell} e^{-\beta(0)U(x)}\,dx,
\]
where we used the monotonicity of $\beta$.
Splitting the integral into the regions $\{|x|\le R\}$ and $\{|x|>R\}$, we have:
\[
  \int_{|x|\le R} |x|^{\ell} e^{-\beta(0)U(x)}\,dx < +\infty
\]
since $U \ge 0$, and, using~\eqref{eq:l-lowerbound},
\[
  \int_{|x|>R} |x|^{\ell} e^{-\beta(0)U(x)}\,dx
  \le \int_{|x|>R} |x|^{-(d+\beta(0)\delta)}\,dx.
\]
In $d$-dimensional spherical coordinates 
(see Theorem~3.4 in~\cite{SteinShakarchi2005}), the latter integral satisfies
\[
  \int_{|x|>R} |x|^{-(d+\beta(0)\delta)}\,dx
  = |S^{d-1}|\int_{R}^{\infty} r^{-(d+\beta(0)\delta)} r^{d-1}\,dr
  = K\,\frac{R^{-\beta(0)\delta}}{\beta(0)\delta} < +\infty,
\]
where the constant $K = |S^{d-1}|$ denotes the surface measure of the 
$(d-1)$-dimensional unit sphere and depends only on the dimension $d$.
 
This proves that $\rho_t \in \mathcal{D}_\ell(\mathbb{R}^d)$.
\end{proof}

\begin{remark}\label{rem:integrability}
As noted in the proof of Lemma~6.1 in~\cite{HuangMalik2025}, a similar argument also shows that, 
under assumption~\eqref{eq:l-lowerbound}, for all $t \ge 0$ and $k \ge 1$,
\[
  \rho_t \in L^k(\mathbb{R}^d),
\]
for any choice of $\ell \in \mathbb{N}^*$.
\end{remark}

\medskip

In the sequel, we first analyze the asymptotic behavior of $\rho_t$, showing that it converges weakly to a limiting measure $\rho_\infty$ supported on the set of global minima of $U$.  
We then establish the existence of a vector field $(v_t)_{t \ge 0}$ satisfying the above continuity property.

\section{Weak convergence of $\rho_t$}
\label{sec:limit}

The goal of this section is to show that the family $(\rho_t)_{t \geq 0}$ asymptotically concentrates around the global minima of the potential $U$ as $t \to +\infty$.  

The main difficulty lies in the fact that the normalization constant $C(t)$, which ensures that $\rho_t$ is a probability density, now depends on time.  
Moreover, this constant appears inside the Lambert function $W_0$, which prevents writing $\rho_t$ as a standard Gibbs-type density and thus makes the classical convergence theorems for such dynamics inapplicable.

\begin{theorem}
\label{th:weak-convergence}
Let $U : \mathbb{R}^d \to \mathbb{R}$ be a $\mathcal{C}^1(\mathbb{R}^d)$ function such that
\[
  \min_{x \in \mathbb{R}^d} U(x) \geq 0,
\]
and let $\beta : [0, +\infty) \to [1, +\infty)$ be a $\mathcal{C}^1$ non-decreasing function such that $\beta(t) \to +\infty$ as $t \to +\infty$.  

Assume furthermore that there exist $R, \delta > 0$ such that
\begin{equation}
  U(y) \geq (\tfrac{d}{\beta(0)} + \delta)\ln|y|,
  \quad \forall\, |y| \geq R.
  \label{eq:Vgrowth-weakconv}
\end{equation}

Then the family $(\rho_t)_{t \geq 0}$ defined by $\rho_t := \rho_{\beta(t)}$ converges weakly, as $t \to +\infty$, to a measure $\rho_\infty$ whose support is contained in the set of global minima of $U$.
\end{theorem}

\medskip

Before proving Theorem~\ref{th:weak-convergence}, it is necessary to study the asymptotic behavior of the normalization constant $C(t)$.  
Indeed, if this constant grows too rapidly, there may exist a neighborhood of the global minima of $U$ where $U(y) - C(t) < 0$, which would compromise the desired convergence.

\begin{lemma}
\label{lem:bound-C}
We have:
\[
  \limsup_{t \to +\infty} C(t) \leq \min_{x \in \mathbb{R}^d} U(x).
\]
\end{lemma}

\begin{proof}
Assume by contradiction that $\limsup_{t \to +\infty} C(t) > \min_{x \in \mathbb{R}^d} U(x)$. The first case is that
\[
   \limsup_{t \to +\infty} C(t) = l,
\]
where $l > \min_{x \in \mathbb{R}^d} U(x)$.  

Let $x_0 \in \mathrm{argmin}\,U(x)$. Then there exist $\varepsilon > 0$, a sequence $(t_n)_{n \geq 0} \subset [0, +\infty)$ and an integer $N \geq 0$ such that
\[
  \forall n \geq N, \quad C(t_n) > U(x_0) + \varepsilon.
\]

For all $n \geq N$, we have
\[
  U(x_0) - C(t_n) < -\varepsilon.
\]
By continuity of $U$, there exists a neighborhood $\mathcal{V}$ of $x_0$ such that
\[
  \forall x \in \mathcal{V},
  \quad U(x)  < U(x_0) +\frac{\varepsilon}{2}.
\]
Hence,
\[
  \forall x \in \mathcal{V}, \ \forall n \geq N, 
  \quad (x) - C(t_n) < -\frac{\varepsilon}{2}.
\]

We can express $\rho_t(x)$ as
\[
  \rho_t(x) = F_t(U(x) - C(t)),
\]
where $F_t$ is a positive and decreasing function of its argument $U(x) - C(t)$.  
Thus, for all $x \in \mathcal{V}$ and $n \geq N$, we have
\[
  \rho_{t_n}(x) > F_{t_n}\!\left(-\tfrac{\varepsilon}{2}\right) > 0.
\]
Integrating over $\mathcal{V}$ gives
\[
  1 
  \geq \int_{\mathcal{V}} \rho_{t_n}(x)\,dx 
  > F_{t_n}\!\left(-\tfrac{\varepsilon}{2}\right)\,\mathrm{Vol}(\mathcal{V}).
\]

However, from the explicit form of $\rho_t$, we have

\[
  F_{t_n}\!\left(-\tfrac{\varepsilon}{2}\right)
  = \left(
      \frac{1}{m}
      W_0\!\left(
        m e^m e^{(m-1)\beta(t_n)\frac{\varepsilon}{2}}
      \right)
    \right)^{\!\frac{1}{m-1}}
  \xrightarrow[n\to\infty]{} +\infty,
\]  

which contradicts the normalization condition
\(
  \int_{\mathbb{R}^d} \rho_t(x)\,dx = 1, \ \forall t \geq 0.
\)

The last case is when $\limsup_{t \to +\infty} C(t) = +\infty$. Then, for $\epsilon > 0$, there exists a sequence $(t_n)_{n \geq 0} \subset [0, +\infty)$ and an integer $N \geq 0$ such that
\[
  \forall n \geq N, \quad C(t_n) > U(x_0) + \varepsilon.
\]
The rest of the proof is similar to the previous case.
\end{proof}

We can now proceed with the proof of Theorem \ref{th:weak-convergence}.

\begin{proof}[Proof of Theorem \ref{th:weak-convergence}]
Consider a sequence $(\rho_{t_n})_{n \geq 0}$. By Lemma \ref{lem:bound-C} we have
\[
  \bar C := \sup_{n\geq 0} C(t_n) < +\infty.
\]
By the coercivity of $U$, there exists $R>0$ such that
\[
  \forall\, |x|\geq R,\qquad U(x)\geq \bar C.
\]

For each $n\geq 0$, recall that, by the explicit form of $\rho_{t_n}$,
\[
  \rho_{t_n}(x)=
  \left(
    \frac{1}{m}
    W_0\!\Big(
      m e^m e^{-(m-1)\beta(t_n)(U(x)-C(t))}
    \Big)
  \right)^{\!\frac{1}{m-1}}.
\]
For $z>0$ we use the inequality $W_0(z)\leq \ln(1+z)\leq z$, so that, for all $t$ and all $|x| \geq R$,
\[
  \rho_{t_n}(x)\leq  e^\frac{m}{m-1} e^{-\beta(t_n)(U(x)-\bar C)}.
\]

Since $\beta(t_n)\geq \beta(0)=:\beta_0$ for all $n\ge0$, we obtain
\[
  \rho_{t_n}(x)\leq e^\frac{m}{m-1}\,e^{-\beta_0 (U(x)-\bar C)}.
\]

From the growth assumption~\eqref{eq:Vgrowth-weakconv}, there exist $R_1\ge R$ and $\delta>0$ such that, for $|x|\ge R_1$,
\[
  \rho_{t_n}(x)\leq e^{\frac{m}{m-1} + \beta_0 \bar C} |x|^{-(d+\beta_0\delta)}.
\]
Thus, using $d$-dimensional spherical coordinates, there exists a constant $K>0$ (independent of $t$) such that, for all $t\ge0$,
\[
  \rho_{t_n}(\{|x|\ge R_1\})
  \le e^{\frac{m}{m-1} + \beta_0 \bar C} \int_{|x|\ge R_1} |x|^{-(d+\beta_0\delta)}\,dx =e^{\frac{m}{m-1} + \beta_0 \bar C} \frac{K}{\delta \beta_0}R^{-\beta_0 \delta}.
\]
Consequently, for any $\varepsilon>0$, there exists $R'>0$ such that
\[
  \forall\, n\ge 0,\qquad \rho_{t_n}(\{|x|\ge R'\})\le \varepsilon.
\]

In other words, the family $(\rho_{t_n})_{n\ge0}$ is \emph{tight}.  
By Prokhorov’s theorem, it is therefore relatively compact for the weak topology.  
Hence, one can extract a subsequence \((\rho_{t_k})_{k\ge0}\) that converges weakly to a probability measure \(\rho_\infty\).  
Let \(M:=\operatorname{argmin}_{\mathbb{R}^d}U\) and consider a bounded open set \(O\subset\mathbb{R}^d\) such that \(\overline{O}\cap M=\varnothing\).

Set
\[
  a:=\inf_{x\in O}\big(U(x)-l\big)>0.
\]

From Lemma \ref{lem:bound-C} we have \(\limsup_{t\to\infty}C(t)\le l\). Hence there exists \(N>0\) such that, for all \(k\ge N\),
\[
  C(t_k)\le l+\frac{a}{2}.
\]
Then, for all \(x\in O\) and \(k\ge N\),
\[
  U(x)-C(t_k)\ge \tfrac{a}{2}.
\]

Therefore, for all \(x\in O\) and \(k\ge N\),
\[
\begin{aligned}
  \rho_{t_k}(x)
  &\le
  \left(\frac{1}{m}\, m e^m e^{-(m-1)\beta(t_k)(U(x)-C(t_k))}\right)^{\!\frac{1}{m-1}}\\
  &= e^{\frac{m}{m-1}}\,
    e^{-\beta(t_k)\big(U(x)-C(t_k)\big)}.
\end{aligned}
\]
Using \(U(x)-C(t_k)\ge \tfrac{a}{2}\), we deduce that there exist constants \(A,L>0\) (here \(A=e^{\frac{m}{m-1}}\) and \(L=\tfrac{a}{2}\)) such that, for all \(k\ge N\) and all \(x\in O\),
\[
  \rho_{t_k}(x)\le A\,e^{- \beta(t_k) L}.
\]

Integrating over \(O\) yields
\[
  \rho_{t_k}(O)=\int_O \rho_{t_k}(x)\,dx
  \le A\,e^{-\beta(t_k) L}\,\mathrm{Vol}(O) \xrightarrow[k \to +\infty]{} 0.
\]
Hence
\[
  \lim_{k\to\infty}\rho_{t_k}(O)=0.
\]

By Portmanteau’s theorem, we have
\[
  \rho_\infty(O)\le \liminf_{k\to\infty}\rho_{t_k}(O)=0,
\]
so that \(\rho_\infty(O)=0\).  
The argument immediately extends to unbounded open sets disjoint from \(M\) by approximating them with an increasing sequence of bounded open sets, which shows that all the mass of \(\rho_\infty\) is concentrated on \(M\).
\end{proof}

\section{Absolute Continuity of the Curve \texorpdfstring{$(\rho_t)_{t \geq 0}$}{(rho\_t)\_t≥0}}
\label{sec:continuity}
\subsection{Motivations and Strategy}

We first present, in a heuristic manner, the motivations and the overall strategy of this section.

Consider the SDE~(\ref{eq:sde-timeinhom}), whose associated Fokker–Planck equation is given by~(\ref{eq:fokker-planck-timeinhom}).  
The diffusion term of this process depends on its marginal law, which makes it a nonstandard McKean–Vlasov-type equation. Such equations are notoriously difficult to analyze rigorously, and only a limited number of results are available in the literature.

Heuristically, one observes that when a trajectory is far from a local minimum, the process behaves like a standard simulated annealing scheme.  
Conversely, when it is trapped in a local minimum, the noise becomes stronger, which should accelerate the escape time from that energy well.  
One may thus conjecture that this process could converge to $\rho_\infty$ for a sufficiently slow cooling schedule $\beta(t)$—at least as slow as that required for classical simulated annealing, i.e., with at most logarithmic speed.

In~\cite{controlledSA}, a control strategy for simulated annealing was proposed to constrain the controlled process to follow a prescribed curve of Gibbs measures, thereby ensuring convergence toward $\rho_\infty$ with a rate depending solely on the choice of the \emph{cooling schedule}.  
Applying this principle to our process~(\ref{eq:sde-timeinhom}) would, on the one hand, provide a theoretical dynamics with faster convergence, and on the other hand, define a process whose convergence is guaranteed without having to study directly the nonstandard McKean–Vlasov-type evolution.

\medskip

The control idea relies on the following observation.  
Let $\mathcal{L}_t^*$ denote the (time-dependent) adjoint operator of the generator associated with~(\ref{eq:sde-timeinhom}). Then the Fokker–Planck equation reads:
\[
  \partial_t \mu_t = \mathcal{L}_t^* \mu_t.
\]
At each fixed time $t \ge 0$, if the temperature parameter $\beta(t)$ is frozen, the process becomes time-homogeneous and admits $\rho_t$ as an invariant measure, i.e., $\mathcal{L}_t^* \rho_t = 0$.

However, the law $\mu_t$ of the process evolves over time, and in general we have $\partial_t \mu_t \neq \mathcal{L}_t^* \rho_t$.  
The goal is therefore to modify the process so that its generator $\bar{\mathcal{L}}_t$ satisfies
\[
  \partial_t \rho_t = \bar{\mathcal{L}}_t^* \rho_t,
\]
so that the law of the process follows exactly the curve $(\rho_t)_{t \ge 0}$.

\medskip

To achieve this, we seek to construct a vector field $(v_t)_{t \ge 0}$ that satisfies the continuity equation:
\[
  \partial_t \rho_t + \nabla \!\cdot\! (\rho_t v_t) = 0,
  \qquad \text{for a.e. } t \geq 0.
\]

By superimposing such a vector field $v_t$ on the original process, the new Fokker–Planck equation becomes:
\[
  \partial_t \rho_t = \mathcal{L}_t^* \rho_t - \nabla \!\cdot\! (\rho_t v_t).
\]
Hence, the problem reduces to identifying a suitable field $(v_t)_{t \geq 0}$, which naturally leads to the framework of optimal transport theory.

\smallskip

Our strategy will therefore be to use the results of Ambrosio–Gigli–Savaré to prove that the curve $(\rho_t)_{t \ge 0}$ is absolutely continuous in the Wasserstein space $(\mathcal{P}_2(\mathbb{R}^d), W_2)$.  
This will then allow us to represent $v_t$ as a limit of an optimal transport problem, providing a tractable expression for the velocity field $v_t$.

\subsection{Main Result}

This section is devoted to the proof of the following result.

\begin{theorem}
\label{thm:existence vt}
Fix \( m > 1 \).  
Let 
\[
  \beta : [0, +\infty) \to [1, +\infty)
\]
be a non-decreasing \( \mathcal{C}^1 \) function such that \( \beta(t) \to +\infty \) as \( t \to +\infty \).  
Let \( U \in \mathcal{C}^1(\mathbb{R}^d) \) be a nonnegative potential whose set of minimizers is nonempty and compact, and assume that \(U\) satisfies conditions~(\ref{eq:gradV-growth}),  and the following growth assumptions:
\begin{equation}
  U(y) \geq (\tfrac{1}{\beta(0)}(d + 4) + \delta)\ln|y|,
  \quad \forall\, |y| \geq R.
  \label{eq:V-lowerbound-strong}
\end{equation}
And,
\begin{equation}
  \exists\, \alpha, R' > 0 \text{ such that } \forall\, |x| \geq R', \quad \langle x, \nabla U(x) \rangle \geq \alpha |x|^2.
  \label{eq:growthV}
\end{equation}
Then, for any \( T \ge 0 \), the curve
\[
\rho: [0,T] \to \mathcal{P}_2(\mathbb{R}^d), \quad t \mapsto \rho_t,
\]
is absolutely continuous.
Moreover, for every \( t \in [0, T] \), there exists a unique minimal-norm vector field \( v_t \in L^2(\mu_t; \mathbb{R}^d) \) such that the pair \( (\rho_t, v_t) \) satisfies the continuity equation in the distributional sense:
\begin{equation}
  \int_{\mathbb{R}^d} \langle v_t(x), \nabla f(x) \rangle \, \rho_t(x)\, dx 
  = \int_{\mathbb{R}^d} f(x)\, \partial_t \rho_t(x)\, dx,
  \quad \forall f \in C_c^\infty(\mathbb{R}^d).
  \label{eq:ce}
\end{equation}
\end{theorem}

\vspace{0.5em}

Before proceeding to the full proof of the theorem, we begin by studying the time regularity of the density \( \rho_t \) and the structure of its time derivative.  
The normalization constant \( C(t) \), ensuring that \( \rho_t \) defines a probability density, now depends on time.  
We shall first assume that \( C \) is differentiable, and later justify that this assumption is indeed valid.

Let \( t > 0 \) and \( x \in \mathbb{R}^d \). We set
\[
  g(t,x) = m e^m e^{-(m-1)\beta(t)(U(x)-C(t))} > 0,
\]
so that
\[
  \rho_t(x) = \left(\frac{1}{m}W_0(g(t,x))\right)^{\frac{1}{m-1}}.
\]

We compute
\[
  \partial_t g(t,x) 
  = \big[-\beta'(t)(m-1)(U(x) - C(t)) + \beta(t)(m-1)C'(t)\big]\,g(t,x).
\]
Moreover, for \( x > 0 \), we recall that 
\[
  W'_0(x) = \frac{1}{x + e^{W_0(x)}}.
\]
Thus,
\[
  \partial_t\!\left(\frac{1}{m}W_0(g(t,x))\right)
  = \frac{1}{m}\,\frac{\partial_t g(t,x)}{g(t,x) + e^{W_0(g(t,x))}}.
\]
Hence,
\begin{align*}
  \partial_t \rho_t(x)
  &= \frac{1}{m}\big[-\beta'(t)(U(x) - C(t)) + \beta(t)C'(t)\big]
  \frac{g(t,x)}{g(t,x) + e^{W_0(g(t,x))}}\,\rho_t(x)^{2-m}\\
  &= \frac{1}{m}\big[-\beta'(t)(U(x) - C(t)) + \beta(t)C'(t)\big]
  \frac{g(t,x)}{g(t,x) + e^{W_0(g(t,x))}}\,\left(\frac{1}{m}W_0(g(t,x))\right)^{-1}\rho_t(x)\\
   &=\big[-\beta'(t)(U(x) - C(t)) + \beta(t)C'(t)\big]
  \frac{1}{1 + W_0(g(t,x))}\,\rho_t(x),
\end{align*}
where we used the elementary identity of the Lambert function:
\[
  W_0(z)e^{W_0(z)}=z,\qquad z>0.
\]
Setting \( a(t,x) := \frac{1}{1+W_0\!\big(g(t,x)\big)} \), we have
\[
\partial_t \rho_t(x) = \big[-\beta'(t)(U(x) - C(t)) + \beta(t)C'(t)\big]
  a(t,x)\rho_t(x).
\]

Before establishing the full proof of Theorem~\ref{thm:existence vt}, it is necessary to show the differentiability of \(C(t)\) and derive an explicit formula for its derivative.

\begin{lemma}
\label{lem:C-differentiable}
For every \(t>0\), let \(C_t\) denote the unique positive constant such that \(\rho_t\) is a probability density.  
Then there exists a function \(C:(0,+\infty)\to\mathbb{R}\) of class \(\mathcal{C}^1\) such that
\[
  C_t = C(t), \quad \forall\, t>0.
\]

Moreover, for all \(t>0\), we have the explicit formula
\begin{equation}
\label{eq:Cprime-expected}
  C'(t) \;=\; \frac{\beta'(t)}{\beta(t)}\,
  \mathbb{E}_{\rho_t^a}\big[U(X)-C(t)\big],
\end{equation}
where the weighted density \(\rho_t^a\) is defined by
\[
  \rho_t^a(x) \;=\; \frac{a(t,x)\,\rho_t(x)}{Z_t^a},\qquad
  Z_t^a:=\int_{\mathbb{R}^d} a(t,x)\,\rho_t(x)\,dx,
\]
and the function \(a(t,x)\) is given by
\[
  a(t,x)
  \;=\; \frac{1}{1+W_0\!\big(g(t,x)\big)},
\]
with
\[
  g(t,x)=m e^m e^{-(m-1)\beta(t)(U(x)-C(t))}.
\]
Finally, there exists a continuous function \(M:(0,+\infty)\to(0,1)\) such that, for all \(t>0\) and \(x\in\mathbb{R}^d\),
\[
  0 < M(t) \le a(t,x) < 1.
\]
\end{lemma}

\begin{proof}
We define the function
\[
F: [0,+\infty) \times \mathbb{R} \longrightarrow (-1,+\infty), \quad (t,C) \longmapsto \int_{\mathbb{R}^d} \rho_t(x;C)\,dx - 1,
\]
where
\[
\rho_t(x;C) := \left( \frac{1}{m} W_0\!\Big(m e^m e^{-(m-1)\beta(t)(U(x)-C)}\Big) \right)^{\frac{1}{m-1}}.
\]
For a fixed time \(t'\), by Corollary 6.1 of~\cite{HuangMalik2025}, there exists a unique \(C^*>0\) such that
\[
F(t', C^*) = 0.
\]
Moreover, the function \(F(t',\cdot)\) is strictly increasing, i.e.,
\[
\partial_C F(t', C) > 0, \quad \forall C>0.
\]
By the implicit function theorem, there exists an open neighborhood \(U\) of \((t', C^*)\) and an open interval \(I \subset (0,+\infty)\) containing \(t'\), as well as a function
\[
C: I \to \mathbb{R}, \quad C \in \mathcal{C}^1(I),
\]
such that
\[
F(t, C(t)) = 0, \quad \forall t \in I.
\]
Furthermore, for all \(t \in I\), we have
\begin{equation}
\label{eq:Cprime}
C'(t) = - \frac{\partial_t F(t, C(t))}{\partial_C F(t, C(t))}.
\end{equation}
Finally, by uniqueness of \(C^*\) for all times \(t\), we can extend the function \(C(\cdot)\) to the interval \((0,+\infty)\) while maintaining the class \(\mathcal{C}^1\).

Thus, for all \(t > 0\),
\begin{align*}
C'(t)
&= -\frac{\displaystyle \int_{\mathbb{R}^d} -\beta'(t)(U(x)-C(t)) \frac{g(t,x)}{g(t,x) + e^{W_0(g(t,x))}} \rho_t(x)^{2-m} \, dx}{\displaystyle \int_{\mathbb{R}^d} \beta(t) \frac{g(t,x)}{g(t,x) + e^{W_0(g(t,x))}} \rho_t(x)^{2-m} \, dx} \\
&= \frac{\beta'(t)}{\beta(t)} \frac{\displaystyle \int_{\mathbb{R}^d} (U(x)-C(t))\, a(t,x)\, d\rho_t(x)}{\displaystyle \int_{\mathbb{R}^d} a(t,x)\, d\rho_t(x)}.
\end{align*}

The equality~\eqref{eq:Cprime-expected} then follows immediately by recognizing the expectation with respect to the probability measure \(\rho_t^a\).

It remains to establish bounds on \(a(t,x)\). Since \(W_0(g) > 0\) for \(g > 0\), we immediately have \(a(t,x) < 1\).

Moreover, the expression of \(g(t,x)\) shows that, for \(x \in \mathbb{R}^d\),
\[
0 < W_0(g(t,x)) \le W_0\Big(m e^m e^{(m-1)\beta(t) C(t)}\Big).
\]  
We then set
\[
M(t) := \frac{1}{1 + W_0\big(m e^m e^{(m-1)\beta(t) C(t)}\big)} > 0.
\]
By construction, \(M(t)\) is strictly positive and, for all \(x\),
\[
a(t,x) \ge M(t).
\]
The continuity of \(M(t)\) in \(t\) follows from the continuity of \(t \mapsto \beta(t)C(t)\) and the continuity of \(W_0\) on \((0,\infty)\).

This completes the proof of the second assertion of the lemma.
\end{proof}

\medskip

Another essential ingredient in the proof relies on a functional property of the measure \(\rho_t\), formulated in the following lemma.

\begin{lemma}
\label{lem:poincare}
Under the same assumptions on \(V\) and \(\beta\) as in Theorem~\ref{thm:existence vt}, let \(T > 0\).  
Then, for all \(t \in [0, T]\), the measure \(\rho_t\) satisfies the following Poincaré inequality:
\[
\mathrm{Var}_{\rho_t}(f) := \int_{\mathbb{R}^d} \left(f - \int f\, d\rho_t \right)^2 \, d\rho_t
\le C_P(\rho_t) \int_{\mathbb{R}^d} |\nabla f|^2 \, d\rho_t, 
\quad \forall f \in C_c^\infty(\mathbb{R}^d).
\]
Moreover, there exists a continuous function
\[
P: [0, T] \to (0,+\infty)
\]
such that, for all \(t \in [0, T]\),
\[
C_P(\rho_t) \le P(t).
\]
\end{lemma}

\begin{proof}
First, observe that
\[
\rho_t(x) = e^{-U_t(x)}, \quad \text{where } U_t(x) = -\ln(\rho_t(x)).
\]
Moreover, by~\eqref{eq:growthV}, there exists \(R' > 0\) such that, for all \(|x| \ge R'\),
\[
\langle x, \nabla U_t(x) \rangle = \left\langle x, \frac{\beta(t)}{1 + W_0(g(t,x))} \nabla U(x) \right\rangle \ge \alpha(t)\, |x|^2,
\]
where we set \(\alpha(t) := \alpha \beta(t) M(t)\). By Lemma~\ref{lem:C-differentiable}, we note that
\[
M(t) \ge \frac{1}{1 + W_0\Big(m e^m e^{(m-1)\beta(T) \max_{[0,T]} C(t)}\Big)} := M_T > 0.
\]
Hence,
\[
0 < \alpha_T := \alpha \beta_0 M_T \le \alpha(t) \le \alpha \beta(T) =: \bar\alpha_T < +\infty.
\]

To prove this lemma, we use Theorem~1.4 of~\cite{BakryBartheCattiauxGuillin2008}.  
Let
\[
L_t = \Delta - \langle \nabla V_t, \nabla \rangle
\]
be the symmetric operator associated with \(\rho_t\).  
To bound the Poincaré constant, it suffices to construct a Lyapunov function \(W : \mathbb{R}^d \to [1,+\infty)\) such that there exist constants \(\theta, b > 0\) and \(R > 0\) satisfying
\[
L_t W(x) \le -\theta W(x) + b \mathbb{1}_{B(0,R)}(x), \quad \forall x \in \mathbb{R}^d.
\]
Under this assumption, one obtains
\[
C_P(\rho_t) \le \theta^{-1} \Big( 1 + b D R^2 e^{\mathrm{Osc}_R(U_t)} \Big),
\]
where \(D > 0\) is a universal constant.

Set
\[
R := \max\{R', \sqrt{2(d+1)/\alpha_T}\}.
\]
Consider
\[
W(x) = e^{\frac{\alpha_T}{4}|x|^2}.
\]
A direct computation gives
\[
L_t W(x) = \frac{\alpha_T}{2} \left( \frac{\alpha_T}{2}|x|^2 + d - \langle x, \nabla U_t(x) \rangle \right) W(x).
\]

For \(|x| \ge R\), we have \(\langle x, \nabla U_t(x) \rangle \ge \alpha(t) |x|^2 \ge \alpha_T |x|^2\), hence
\[
L_t W(x) \le -\frac{\alpha_T}{2} \left( \frac{\alpha_T}{2} |x|^2 - d \right) W(x) \le -\frac{\alpha_T}{2} \left( \frac{\alpha_T}{2} R^2 - d \right) W(x).
\]
We can thus set
\[
\theta := \frac{\alpha_T}{2} \left( \frac{\alpha_T}{2} R^2 - d \right) > 0.
\]

On the ball \(B(0,R)\), using \(|\nabla U(x)| \le K(1+|x|)\), one obtains
\[
L_t W(x) \le -\theta W(x) + \frac{\alpha_T}{2} \Big( \alpha_T R^2 + \bar\alpha_T R K (1+K^2) \Big) e^{\frac{\alpha_T R^2}{4}} \, \mathbb{1}_{B(0,R)}(x).
\]
Hence, one can set
\[
b := \frac{\alpha_T}{2} \Big( \alpha_T R^2 + \bar\alpha_T R K (1+K^2) \Big) e^{\frac{\alpha_T R^2}{4}} > 0.
\]

Thus, \(\rho_t\) indeed satisfies a Poincaré inequality, with
\[
C_P(\rho_t) \le \theta^{-1} \Big( 1 + b D R^2 e^{\mathrm{Osc}_{B(0,R)}(U_t)} \Big),
\]
where \(\theta, b > 0\) depend only on \(T\) and \(d\).

Finally, \(\mathrm{Osc}_{B(0,R)}(U_t) := \sup_{x \in B(0,R)} U_t(x) - \inf_{x \in B(0,R)} U_t(x)\) is finite and depends continuously on \(t\) since \(U_t\) is continuous on the compact set \([0,T] \times \overline{B(0,R)}\). This proves the second part of the lemma.
\end{proof}

We now have all the tools to prove Theorem~\ref{thm:existence vt}.

\begin{proof}
By Lemma~\ref{lem:C-differentiable}, we have
\begin{align}
  \partial_t \rho_t(x)
  &= -\beta'(t)\Big[(U(x) - C(t))
      - \mathbb{E}_{\rho_t^a}[U(x) - C(t)]\Big]
      \rho_t^a(x) Z_t^a\\
  &= -\beta'(t)\Big[U(x)
      - \mathbb{E}_{\rho_t^a}[U(x)]\Big]
      \rho_t^a(x),
\end{align}

and \(0 < Z_t^a < 1\), hence
\begin{equation}
  \frac{M(t)}{Z_t^a} \rho_t(x)\leq\rho_t^a(x) \le \frac{1}{Z_t^a} \rho_t(x).
  \label{eq:encadrement}
\end{equation}

We define, for all \(t > 0\), the linear map
\[
  \Phi : \dot{H}(\rho_t) \longrightarrow \mathbb{R}, 
  \qquad 
  \Phi(f) = \int_{\mathbb{R}^d} f(x)\,\partial_t \rho_t(x)\,dx,
\]
where 
\[
  \dot{H}(\rho_t)
  = \Big\{
      f \in L^2(\rho_t) : 
      \nabla f \in L^2(\rho_t; \mathbb{R}^d),
      \ \int_{\mathbb{R}^d} f(x)\,\rho_t(x)\,dx = 0
    \Big\}.
\]

From the obtained expression for \(\partial_t \rho_t\), we have
\[
  \Phi(f)
  = -\beta'(t) Z_t^a
    \int_{\mathbb{R}^d}
      f(x)
      \Big[
        U(x)
        - \mathbb{E}_{\rho_t^a}[U(X)]
      \Big]
      \rho_t^a(x)\,dx.
\]
By the Cauchy–Schwarz inequality,
\[
  |\Phi(f)|
  \leq |\beta'(t)| Z_t^a 
  \sqrt{\int_{\mathbb{R}^d} |f(x)|^2 \rho_t^a(x)\,dx}
  \sqrt{\mathrm{Var}_{\rho_t^a}[U(X)]}.
\]

and thus, using (\ref{eq:encadrement}),
\[
  |\Phi(f)|
  \leq |\beta'(t)| \sqrt{Z_t^a} 
  \sqrt{\mathrm{Var}_{\rho_t}[f(X)]}
  \sqrt{\mathrm{Var}_{\rho_t^a}[U(X)]}.
\]

Moreover,
\[
  \mathrm{Var}_{\rho_t^a}[U(X)]
  \le \mathbb{E}_{\rho_t^a}[U^2(X)]
  \le \frac{1}{Z_t^a}\,\mathbb{E}_{\rho_t}[U^2(X)].
\]
Now, by \eqref{eq:gradV-growth}, there exists a constant \(C_K > 0\) such that
\[
  |U(x)| \le C_K(1 + |x|^2),
\]
which implies
\[
  \mathbb{E}_{\rho_t}[U^2(X)]
  \le C_K^2 \int_{\mathbb{R}^d} (1 + |x|^2)^2 \rho_t(x)\,dx
  < +\infty,
\]
since, by \eqref{eq:V-lowerbound-strong}, \(\rho_t\) has a finite fourth moment.

We finally obtain
\[
  |\Phi(f)|
  \le |\beta'(t)|
  \sqrt{\mathrm{Var}_{\rho_t}[f(X)]}
  \sqrt{\mathbb{E}_{\rho_t}[U^2(X)]},
  \qquad \text{with }
  \mathbb{E}_{\rho_t}[U^2(X)] < +\infty.
\]

We now estimate the dual norm of \( \phi \) on \( \dot{H}^1(\mu_t) \). By definition:
\begin{align}
\|\phi\|_{\dot{H}^1(\rho_t)^*} 
&= \sup_{f \in \dot{H}^1(\rho_t)} \frac{|\phi(f)|}{\|\nabla f\|_{L^2(\rho_t)}}\\ 
&\leq |\beta'(t)| \sup_{f \in \dot{H}^1(\rho_t)} \frac{ \sqrt{\mathrm{Var}_{\rho_t}(f)} \cdot \sqrt{ \mathbb{E}_{\rho_t}[U^2(X)]} }{ \|\nabla f\|_{L^2(\rho_t)} } \notag \\
&\leq |\beta'(t)| \sqrt{C_P(\rho_t)} \cdot \sqrt{ \mathbb{E}_{\rho_t}[U^2(X)]}\\ 
&\leq |\beta'(t)| \sqrt{P(t)} \cdot \sqrt{ \mathbb{E}_{\rho_t}[U^2(X)]} < \infty,
\end{align}
where we used the Poincaré inequality and Lemma~\ref{lem:poincare}.

\medskip

Therefore, \( \phi \) is a continuous linear form on the Hilbert space \( (\dot{H}^1(\rho_t), \langle \cdot, \cdot \rangle_{\nabla,\rho_t}) \). By the Riesz representation theorem, there exists a unique \( h_t \in \dot{H}^1(\rho_t) \) such that:
\begin{equation}
\phi(f) = \langle h_t, f \rangle_{\nabla\rho_t}, \quad \forall f \in \dot{H}^1(\rho_t),
\end{equation}
and
\[
\|\phi\|_{\dot{H}^1(\rho_t)^*} = \|h_t\|_{\dot{H}^1(\rho_t)}.
\]

In particular, for all \( f \in C_c^\infty(\mathbb{R}^d) \cap \dot{H}^1(\rho_t) \), we have:
\[
\int f(x) \, \partial_t \rho_t(x)\, dx = \int \langle \nabla h_t(x), \nabla f(x) \rangle \, \rho_t(x)\, dx.
\]

Now observe that for any \( f \in C_c^\infty(\mathbb{R}^d) \), we can write \( f = (f - \mathbb{E}_{\rho_t}[f]) + \mathbb{E}_{\rho_t}[f] \), and since \( \nabla \mathbb{E}_{\rho_t}[f] = 0 \), the above equation extends to all \( f \in C_c^\infty(\mathbb{R}^d) \).

\medskip

Hence, setting \( v_t := \nabla h_t \), we obtain a solution to the continuity equation in the prescribed sense \eqref{eq:ce}.

Moreover, we have:
\[
\|v_t\|_{L^2(\rho_t)} = \|\nabla h_t\|_{L^2(\rho_t)} = \|h_t\|_{\dot{H}^1(\rho_t)} = \|\phi\|_{\dot{H}^1(\rho_t)^*} \leq |\beta'(t)| \sqrt{P(t)} \cdot \sqrt{\mathrm{Var}_{\rho_t}[U(X)]}.
\]

Since \( \sup_{t \in [0,T]} |\beta'(t)| \cdot \sqrt{P(t)} \cdot \sqrt{\mathrm{Var}_{\rho_t}[U(X)]} < \infty \), we conclude that \( t \mapsto \|v_t\|_{L^2(\rho_t)} \in L^1([0,T]) \). An application of Theorem~\ref{thm:AGS_continuity} yields that \( \rho : t \mapsto \rho_t \) is absolutely continuous in \( \mathcal{P}_2(\mathbb{R}^d) \).

\medskip

Finally, since \( v_t \in \overline{ \{ \nabla \varphi : \varphi \in C_c^\infty(\mathbb{R}^d) \} }^{L^2(\rho_t)} \), an application of Theorem~\ref{thm:minimal_velocity_field} shows that \( v_t \) is the unique velocity field of minimal norm satisfying the continuity equation. This completes the proof.
\end{proof}

\section{Well-posedness of the Controlled Swarm Gradient Dynamics}
\label{sec:controlled}

After establishing the existence of a velocity field $(v_t)_{t\in[0,T]}$ satisfying the continuity equation along the prescribed curve of measures $(\rho_t)_{t\in[0,T]}$, we are now in a position to construct a stochastic process whose time marginals coincide with this curve. 
More precisely, we study the well-posedness of the following controlled stochastic differential equation:
\begin{equation}\label{eq:c-sde}
\mathrm{d}X_t
=
\big(v_t(X_t) - \nabla U(X_t)\big)\,\mathrm{d}t
+
\sqrt{\frac{2}{\beta(t)}\,\alpha(\rho(t,X_t))}\,\mathrm{d}B_t,
\qquad X_0 \sim \rho_0,
\end{equation}
where
\[
\alpha(r) = 1 + r^{m-1},
\]
and where the density $\rho_t$ is given explicitly by
\[
\rho(t,x)
=
\left(
  \frac{1}{m}\,W_0\!\Big(m e^m e^{-(m-1)\beta(t)(U(x)-C(t))}\Big)
\right)^{\!\frac{1}{m-1}}.
\]

\medskip

It is important to emphasize that the SDE \eqref{eq:c-sde} is now \emph{linear in law}. 
Indeed, the function $\rho(t,x)$ plays the role of a prescribed coefficient and is no longer the marginal law of $X_t$.
This stems from our construction: the goal is to define a process whose law at each time $t$ is \emph{explicitly imposed} to be $\rho_t$, thereby removing the McKean--Vlasov-type nonlinearity.

The well-posedness of \eqref{eq:c-sde} essentially relies on regularity and growth assumptions on the velocity field $v_t$, for which only limited information is available.
To simplify the analysis, we adopt the same assumption as in~\cite{controlledSA}, namely that $v_t$ has at most linear growth, which is the same assumption imposed on $\nabla U$.

\begin{proposition}\label{prop:wellposed-csde}
Assume the hypotheses of Theorem~\ref{thm:existence vt} and let $(v_t)_{t\in[0,T]}$ be the velocity field constructed therein.
Assume moreover that
\begin{equation}\label{eq:vt-linear-growth}
\exists\, C>0 \quad \text{such that} \quad
|v_t(x)| \le C(1+|x|),
\qquad \forall (t,x)\in[0,T]\times\mathbb{R}^d.
\end{equation}
Then the SDE \eqref{eq:c-sde} admits at least one weak solution $(X_t)_{t\in[0,T]}$.
Moreover, for any weak solution, we have
\[
\rho^{X_t}=\rho_t(dx)\,
\qquad \forall\, t\in[0,T], \quad \text{a.s.}
\]
\end{proposition}

\begin{proof}
For $(t,x)\in[0,T]\times\mathbb{R}^d$, define the drift and diffusion coefficients
\[
b(t,x) := v_t(x)-\nabla U(x),
\qquad
A(t,x) := a(t,x)\,I_d,
\quad
a(t,x) := \beta(t)^{-1}\alpha(\rho(t,x)).
\]
Consider the following Cauchy problem:
\begin{equation}\label{eq:cauchy}
\partial_t \mu_t
=
\mathcal{L}^*\mu_t
=
\nabla\cdot\!\big(A(t,\cdot)\nabla\mu_t\big)
-
\nabla\cdot\!\big(b(t,\cdot)\mu_t\big),
\qquad
\mu_0=\rho_0,
\end{equation}
where $\mathcal{L}^*$ denotes the formal adjoint of the differential operator
\[
\mathcal{L}u
=
a(t,x)\Delta u
+
\langle b(t,x),\nabla u\rangle.
\]

If a solution $(X_t)_{t\in[0,T]}$ to \eqref{eq:c-sde} exists, then the laws of its marginals must satisfy \eqref{eq:cauchy}, that is, the associated Fokker--Planck equation.
We first observe that $(\rho_t)_{t\in[0,T]}$ indeed solves \eqref{eq:cauchy}.
On the one hand, by Theorem~\ref{thm:existence vt}, the pair $(\rho_t,v_t)$ satisfies the continuity equation
\[
\partial_t\rho_t + \nabla\cdot(\rho_t v_t)=0.
\]
On the other hand, for each fixed $t$, $\rho_t$ is the stationary distribution of the time-homogeneous swarm gradient dynamics at inverse temperature $\beta = \beta(t)$, which implies
\[
\nabla\cdot(\rho_t\nabla U)+\nabla\cdot(A(t,\cdot)\nabla\rho_t)=0.
\]
Combining both identities shows that $\rho_t$ is a weak solution of \eqref{eq:cauchy}.

We now apply the Ambrosio--Figalli--Trevisan superposition principle to construct a stochastic process with prescribed marginals $(\rho_t)_{t\in[0,T]}$.
To this end, it suffices to verify that
\[
\int_0^T\!\!\int_{\mathbb{R}^d}
\frac{\|A(t,x)\| + |\langle b(t,x),x\rangle|}
     {(1+|x|)^2}\,\rho_t(\mathrm{d}x)\,\mathrm{d}t < \infty.
\]

By assumption \eqref{eq:vt-linear-growth}, there exists $C_1>0$ such that
\[
\frac{|\langle v_t(x),x\rangle|}{(1+|x|)^2}
\le
\frac{C_1(1+|x|)^2}{(1+|x|)^2}
= C_1.
\]
Similarly, by the growth assumption on $\nabla U$ \eqref{eq:gradV-growth}, there exists $C_2>0$ such that
\[
\frac{|\langle\nabla U(x),x\rangle|}{(1+|x|)^2} \le C_2.
\]
Finally,
\[
\frac{\|A(t,x)\|}{(1+|x|)^2}
\le
\beta(0)^{-1}\big(1+\rho(t,x)^{m-1}\big)
\le
\beta(0)^{-1}(1+C_T),
\]
where
\[
C_T
:=
\frac{1}{m}
W_0\!\Big(
m e^m e^{(m-1)\beta(T)\sup_{t\in[0,T]} C(t)}
\Big).
\]
Therefore,
\[
\int_0^T\!\!\int_{\mathbb{R}^d}
\frac{\|A(t,x)\| + |\langle b(t,x),x\rangle|}
     {(1+|x|)^2}\,\rho_t(\mathrm{d}x)\,\mathrm{d}t
\le
T\big(\beta(0)^{-1}(1+C_T)+C_1+C_2\big)<\infty.
\]

The Ambrosio--Figalli--Trevisan superposition principle \cite[Theorem~1.1]{BogachevRocknerShaposhnikov2021} therefore applies and yields the existence of a Borel probability measure $\mathbb{P}_{\rho_0}$ on $\Omega=C([0,T],\mathbb{R}^d)$, equipped with its natural filtration, such that $\mathbb{P}_{\rho_0}$ solves the martingale problem associated with \eqref{eq:c-sde}.
Moreover, under $\mathbb{P}_{\rho_0}$, we have
\[
\rho^{X_t}=\rho_t(x)\,\mathrm{d}x,
\qquad \forall t\in[0,T].
\]
By \cite[Theorem~32.7]{Kallenberg2021}, , this implies that \eqref{eq:c-sde} admits a weak solution with law $\mathbb{P}_{\rho_0}$.

It remains to prove uniqueness of the marginal laws.
To this end, we show that the Fokker--Planck equation \eqref{eq:cauchy}
admits a unique probability solution.
It suffices to verify that the assumptions of
\cite[Theorem~9.4.3]{BogachevKrylovRocknerShaposhnikov2022}
are satisfied in the present setting.

First, the growth assumptions \eqref{eq:gradV-growth} and
\eqref{eq:vt-linear-growth} imply that, for any $p>0$,
\[
b \in L^p_{\mathrm{loc}}\big(\mathbb{R}^d\times(0,T)\big).
\]
Moreover, by an argument similar to the one used in the application of the
Ambrosio--Figalli--Trevisan superposition principle, we have
\[
\int_0^T\!\!\int_{\mathbb{R}^d}
\left(
\frac{\|A(t,x)\|}{1+|x|^2}
+
\frac{|\langle b(t,x),x\rangle|}{1+|x|}
\right)
\,\rho_t(\mathrm{d}x)\,\mathrm{d}t
< \infty.
\]

It remains to verify that for each ball $U\subset\mathbb{R}^d$ there exist constants $\gamma,M,\Lambda>0$ such that,
for all $(x,y,t)\in U\times U\times(0,T)$,
\[
\gamma I_d \le A(t,x),
\qquad
\|A(t,x)\|\le M,
\qquad
|a(t,x)-a(t,y)|\le \Lambda |x-y|.
\]

Set $f_t(x):=\rho(t,x)^{m-1}$.
Since $U\in C^1(\mathbb{R}^d)$, we have $f_t\in C^1(\mathbb{R}^d)$.
By the finite increment inequality, for all $x,y\in U$,
\[
|a(t,x)-a(t,y)|
\le
\beta(0)^{-1}|f_t(x)-f_t(y)|
\le
\beta(0)^{-1}\sup_{z\in U}|\nabla f_t(z)|\,|x-y|.
\]

Therefore, choosing
\[
\gamma := \beta(T)^{-1},
\qquad
M := \beta(0)^{-1}(1+C_T),
\qquad
\Lambda := \beta(0)^{-1}\sup_{z\in U}|\nabla f_t(z)|,
\]
all the assumptions of
\cite[Theorem~9.4.3]{BogachevKrylovRocknerShaposhnikov2022}
are satisfied.

Hence, the Fokker--Planck equation \eqref{eq:cauchy} admits a unique probability solution.
As a consequence, any weak solution of \eqref{eq:c-sde} has time marginals given by
$\rho_t$, for all $t\in[0,T]$, almost surely.

\end{proof}

\section{Numerical implementation of the controlled process}
\label{sec:implementation}

\subsection{Implementation strategy}

We recall the controlled dynamics studied in this paper:
\begin{equation}\label{eq:controlled-sde}
\mathrm{d}X_t
=
\big(v_t(X_t) - \nabla U(X_t)\big)\,\mathrm{d}t
+
\sqrt{\frac{2}{\beta(t)}\,\alpha(\rho_t(X_t))}\,\mathrm{d}B_t,
\qquad X_0 \sim \rho_0,
\end{equation}
where
\[
\alpha(r) = 1 + r^{m-1},
\qquad
\rho_t(x) = \rho(x,t,C(t)),
\]
and the explicit parametric form of the density is
\[
\rho(x,t,C)
=
\left(
  \frac{1}{m}\,W_0\!\Big(m e^m e^{-(m-1)\beta(t)(U(x)-C)}\Big)
\right)^{\!\frac{1}{m-1}}.
\]

An important observation about the controlled process is that \emph{at every time $t$ we know its marginal $\rho_t$ explicitly} (up to the normalization constant $C(t)$). Consequently, the dynamics is no longer nonlinear in law in the sense of McKean–Vlasov: the dependence on the law can be replaced by the explicit function $\rho(\cdot,t,C(t))$. This observation has a major practical implication: unlike the uncontrolled McKean–Vlasov dynamics where an algorithmic implementation typically requires repeated, costly density estimation (e.g.\ KDE) of the marginal law, the controlled scheme can avoid explicit density estimation by instead estimating the velocity field \(v_t\) (equivalently the Monge map from \(\rho_t\) to \(\rho_{t+h}\)) so that the particle system remains close to the prescribed curve \((\rho_t)_{t\ge 0}\).

A remaining difficulty is that the density $\rho_t$ depends explicitly on the normalization constant \(C(t)\). One can simply estimate $C(t)$ from a sample of $\rho_t$ by using its characterization as the root of
\[
C \;\mapsto\; \frac{1}{N} \sum_{i=1}^N \rho(t, X_t^i, C) - 1,
\]
computed numerically using a one-dimensional root-finding method (such as Brent's method).

Concretely, let \(\{X^i_t\}_{i=1}^N\) denote the particle system approximating \(X_t\). Over a time step \(h>0\) one can

\begin{enumerate}
  \item Estimate $C(t)$ as described above.
  \item Update particles by an Euler–Maruyama step using the current estimate \(v_t\) and the current estimated density parameters:
  \[
    X^i_{t+h}
    = X^i_t + h\big(v_t(X^i_t)-\nabla U(X^i_t)\big)
    + \sqrt{\frac{h}{\beta(t)}\,\alpha(\rho(t, X^i_t, C(t)))}\,\xi^i_t,
  \]
  where \(\xi^i_t \overset{\mathrm{iid}}{\sim} \mathcal{N}(0,I_d)\).
\end{enumerate}

\medskip

It suffices to draw an initial sample from \(\rho_0\) by running the uncontrolled time-homogeneous process at temperature \(\beta(0)\) to start the discretization algorithm.

\subsection{Approximation of the Velocity Field \(v_t\)}

The objective is to obtain an approximation of the velocity field \(v_t\). 
In general, such a vector field is intractable to compute explicitly. 
Since the family \((\rho_t)_{t \geq 0}\) is absolutely continuous, Proposition~\ref{prop:charac_vt} suggests that \(v_t\) can be approximated arbitrarily well by computing the Monge map 
\(T_{\rho_t \to \rho_{t+h}}\) for some \(h > 0\), through the characterization
\[
v_t \approx \frac{T_{\rho_t \to \rho_{t+h}} - \mathrm{id}}{h}.
\]

However, because the constant \(C(t)\) evolves in parallel with the sample \(\{X_t^i\}_{i=1}^N\) from \(\rho_t\), it is coherent to assume access to \(\rho_t\) at time \(t\), but not to \(\rho_{t+h}\).  

First, we propose a strategy to obtain a coarser estimate of \(C(t+h)\) when only a sample from \(\rho_t\) is available. For this, we use the characterization of \(C'(t)\) given by Lemma~\ref{lem:C-differentiable}. Therefore, \(C(t+h)\) can be estimated by the coarse estimate
\[
C_{t+h}^v = C(t) + h\, \widehat C'(t),
\]
where \(\widehat C'(t)\) is obtained from the empirical expression of \(C'(t)\):
\[
\widehat{C}'(t)
\;=\;
\frac{\beta'(t)}{\beta(t)}\,
\frac{\sum_{i=1}^N \big(U(X^i_t)-C(t)\big)\,a(t,X^i_t,C(t))}
     {\sum_{i=1}^N a(t,X^i_t,C(t))}.
\]

Then, to approximate \(v_t\) we use the strategy proposed in~\cite{controlledSA}, which aims to control the simulated annealing dynamics so that it follows a prescribed Gibbs measure path, is to transform the continuous optimal transport problem into a discrete one. 
To do so, they use an importance sampling technique to reweight the empirical measure of the sample at time \(t\), initially distributed according to \(\mu_t\), so as to obtain a discrete approximation of \(\mu_{t+h}\).

\medskip

We can apply this approach to our setting. 
The only difference with the original Gibbs framework, where one has access to the unnormalized density, is that in our case, we only have access to the density \(\rho_t\) up to the approximation \(C(t)\) of the true parameter \(C_t\). 
Therefore, we approximate \(\rho_t\) by the normalized density
\[
\rho_t(x) \approx \frac{\rho(x, t, C(t))}{Z_t},
\]
where \(Z_t\) is the normalizing constant, and we use the same self-normalized importance sampling strategy for the unnormalized densities 
\(\rho(x, t+h, C(t))\) and \(\rho(x, t+h, C(t+h))\).

\medskip

An additional practical consideration is that, whereas the Euler–Maruyama discretization requires a small time step \(\Delta t\) to avoid deviating too far from the true process, estimating 
\(T_{\rho_t \to \rho_{t+h}}\) at such a frequency would be computationally prohibitive. 
We therefore introduce two time scales:
a fine discretization step \(\Delta t\) for the particle dynamics, and a coarser update interval \(h = k \Delta t\) for the estimation of \(v_t\).
At each velocity update, we compute a rougher estimate of \(C(t+h)\), denoted \(C_{t+h}^v\), which is used only for the velocity field estimation, 
while a smoother update of \(C(t)\) with step \(\Delta t\) is used within the SDE discretization.

\medskip

In practice, the reweighted empirical measure at time \(t+h\) is defined as
\[
\hat\rho_{t+h} := \sum_{i=1}^n \omega_i \, \delta_{x_i}, 
\qquad 
\tilde{\omega}_i := 
\frac{\rho(X_t^i, t+h, C(t+h))}{\rho(X_t^i, t, C(t))},
\qquad 
\omega_i = \frac{\tilde{\omega}_i}{\sum_j \tilde{\omega}_j}.
\]

Once the discrete measures 
\(\hat\rho_t = \frac{1}{n}\sum_i \delta_{X_t^i}\) and 
\(\hat\rho_{t+h} = \sum_i \omega_i \delta_{X_t^i}\)
are defined, we estimate the transport map \(T_{\rho_t \to \rho_{t+h}}\)
by solving the discrete optimal transport problem between these two empirical measures.
Let \(C_{ij} = \|X_t^i - X_t^j\|^2\) denote the cost matrix, 
and let \(G^*\) be the optimal coupling solving
\[
\min_{G \ge 0} \ \langle C, G \rangle_{\mathrm{Fr}}
\quad \text{s.t. } \quad G\mathbf{1} = \mathbf{1}, \quad G^\top \mathbf{1} = n\mathbf{w}.
\]
The \emph{barycentric projection} provides an approximation of the Monge map by associating to each source particle \(X_t^i\) the barycenter of its transported mass:
\[
T(X_t^i) := \sum_{j=1}^n G^*_{ij} X_t^j.
\]
The corresponding velocity estimator then reads
\[
V_t^i = \frac{T(X_t^i) - X_t^i}{h}.
\]

\subsection{Discretization Scheme for the Controlled Swarm Gradient Process}

We can now propose an implementable algorithmic scheme to approximate the controlled process~\eqref{eq:controlled-sde}.

\begin{algorithm}[H]
\caption{Controlled Swarm Gradient Dynamics}
\begin{algorithmic}[1]
\State \textbf{Input:} Objective $U$, cooling schedule $\beta$, initial states $\mathbf{X}_0 = (X^i_0)_{i=1}^n$, initial normalization constant $C_0$, time step $\Delta t$, update interval $h = k \Delta t$, final time $T$.
\State Set $t = 0$.
\While{$t < T$}

  \Statex \textbf{1. Estimate the velocity field:}
   \State Compute the coarse update for the velocity: $C_{t+h}^v = C_t + h \, \hat{C}'(t)$.
  \State Compute the weights $\tilde{w}_i = \dfrac{\rho(X_t^i, t+h, C_{t+h}^v)}{\rho(X_t^i, t, C_t)}$.
  \State Normalize: $w_i = \tilde{w}_i / \sum_j \tilde{w}_j$.
  \State Define the cost matrix $C_{ij} = |X_t^i - X_t^j|^2$.
  \State Solve the optimal transport problem:
  \[
  G^* := \arg\min_{G \ge 0} \langle C, G \rangle_{\mathrm{Fr}}
  \quad \text{s.t. } G \mathbf{1} = \mathbf{1}, \quad G^\top \mathbf{1} = n \mathbf{w}.
  \]
  \State Compute the barycentric projection $T(X_t^i) = \sum_j G^*_{ij} X_t^j$.
  \State Set $V_t^i = \dfrac{T(X_t^i) - X_t^i}{h}$ for $i = 1, \dots, n$.

  \Statex \textbf{2. Integrate the SDE:}
  \For{$\ell = 0$ to $k-1$}
    \State Sample $\xi_\ell^i \sim \mathcal{N}(0, I_d)$ independently.
    \State Update:
    \[
    X^i_{t + (\ell+1)\Delta t} = 
    X^i_{t + \ell \Delta t} 
    + \Delta t \left(V_t^i - \nabla U(X^i_{t+\ell\Delta t})\right)
    + \sqrt{\beta^{-1}(t + \ell\Delta t)\,
      \alpha(\rho(X^i_{t + \ell \Delta t}, t+\ell\Delta t, C_{t+\ell\Delta t}))}\,
      \Delta t^{1/2} \xi_\ell^i.
    \]
    \State Update the normalization constant:
    \[
        C_{t + (\ell+1)\Delta t}
        = \arg\!\left\{ C : \frac1n \sum_{i=1}^n \rho\big(t + (\ell+1)\Delta t, X^i_{t+(\ell+1)\Delta t}, C\big) = 1 \right\}.
    \]
  \EndFor
  \State Update $t \gets t + h$.
\EndWhile
\end{algorithmic}
\end{algorithm}

\begin{remark}[Initialization trick]
In principle, this type of controlled process is initialized by drawing the
initial sample $X_0$ from the invariant distribution of the time-homogeneous,
uncontrolled dynamics corresponding to a fixed initial inverse temperature
$\beta(0)$. In practice, this preliminary sampling phase can be computationally
costly.

A simple initialization trick allows one to bypass this step.
When working with potentials $U$ having compact support, one may start from an
initial distribution that is easy to sample and supported on the same compact
set, typically the uniform distribution, and choose a cooling schedule such
that $\beta(0)=0$.
In this case, the first velocity field can be obtained by estimating a single
optimal transport map between the uniform distribution and the target density
$\rho^{X_h}$ at the first time step.

This strategy can also be employed for coercive potentials defined on
$\mathbb{R}^d$.
Indeed, choosing a sufficiently large compact set, particles are very unlikely
to leave this region due to coercivity, so that initializing from a compactly
supported distribution does not significantly affect the performance of the
algorithm.

This initialization trick is particularly effective for controlled simulated
annealing.
First, the uniform distribution is a natural initial law, since the Gibbs
measure $\mu_\beta$ converges toward the uniform distribution as
$\beta \to 0$.
Second, in this setting, the estimation of the control field $v_t$ via
importance sampling is independent of the normalization constant, and depends
only on the density of $\mu_h$.

By contrast, this strategy is less effective for controlled swarm gradient
dynamics.
Indeed, the estimation of the velocity field $v_t$ in this case crucially
depends on an accurate estimation of the normalization constant.
An inaccurate initialization of $C_h$ may therefore propagate through time and
significantly deteriorate the numerical performance of the controlled swarm
algorithm.
\end{remark}

\section{Numerical Results}
\label{sec:numerical}

In this section, we present the numerical experiments carried out to evaluate the performance of controlled swarm gradient dynamics (CSG) algorithms compared with that of controlled simulated annealing (CSA).

\subsection{1D Double-Well Function}

We first tested the controlled swarm gradient on the one-dimensional double-well function [Figure~\ref{fig:double_well}] introduced in~\cite{GaoXuZhou2022} and used it to test uncontrolled swarm gradient dynamics in~\cite{HuangMalik2025} . The potential is defined as
\begin{equation}
U(x) = 
\begin{cases}
-12 x - 52, & x \le -6, \\
2(x+3)^2 + 2, & -6 < x < -2, \\
8 - x^2, & -2 \le x \le 2, \\
(x-4)^2, & 2 < x \le 6, \\
4 x - 20, & x > 6.
\end{cases}
\end{equation}
This function has a global minimum at \(x^* = 4\) and a local minimum at \(x^* = -3\), separated by a significant energy barrier. 

\begin{figure}[h!]
\centering
\includegraphics[width=0.6\textwidth]{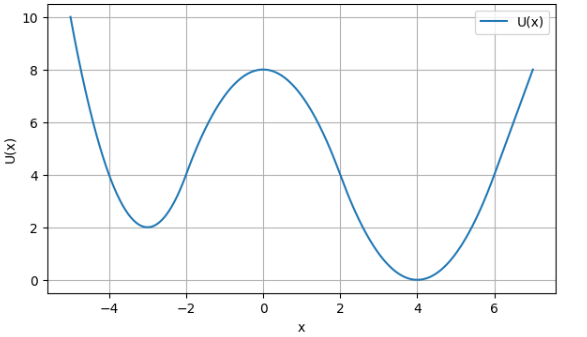}
\caption{1D double-well potential \(U(x)\).}
\label{fig:double_well}
\end{figure}

We implemented both the controlled swarm gradient and the controlled simulated annealing, defined by
\begin{equation}
\mathrm{d}X_t = \big(v_t(X_t) - \nabla U(X_t)\big)\,\mathrm{d}t + \sqrt{\frac{2}{\beta(t)}}\,\mathrm{d}B_t, \qquad X_0 \sim \mu_0
\end{equation}
where \(v_t\) satisfies the continuity equation with the target Gibbs measure \((\mu_t)_{t \ge 0}\), \(\mu_t(x) \propto e^{-\beta(t) U(x)}\).

Each process was initialized by sampling from \(\rho_0\) (or \(\mu_0\)) using an uncontrolled swarm gradient with empirical marginal estimated via Gaussian KDE (or a Langevin diffusion at fixed \(\beta(0)\)). The temperature schedule was chosen quadratic:
\begin{equation}
\beta(t) = 0.25 + 25\, t^2.
\end{equation}

\subsubsection{Particle Heatmaps}

We present four heatmaps of particle distributions over time [Figure~\ref{fig:heatmaps_doublewell}], aggregated from 100 runs with 100 initial particles sampled from \(\rho_0\) (or \(\mu_0\)) each. The time interval is \(t \in [0,1]\) with \(500\) iterations (\(\Delta t = 0.002\)), and \(v_t\) is estimated every 20 iterations (\(h = 0.04\)) except for the last experiment where \(h = 0.02\). 

\begin{figure}[h!]
\centering
\includegraphics[width=0.48\textwidth]{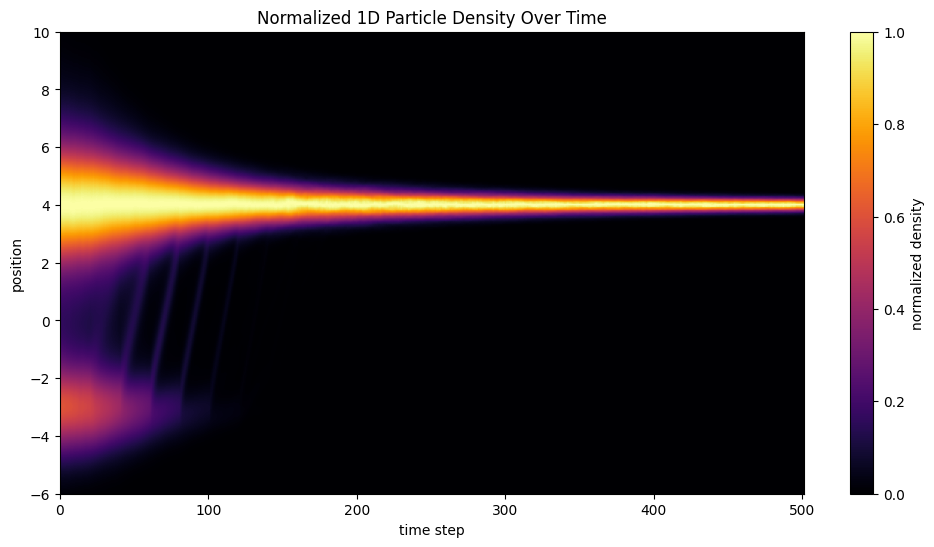}
\includegraphics[width=0.48\textwidth]{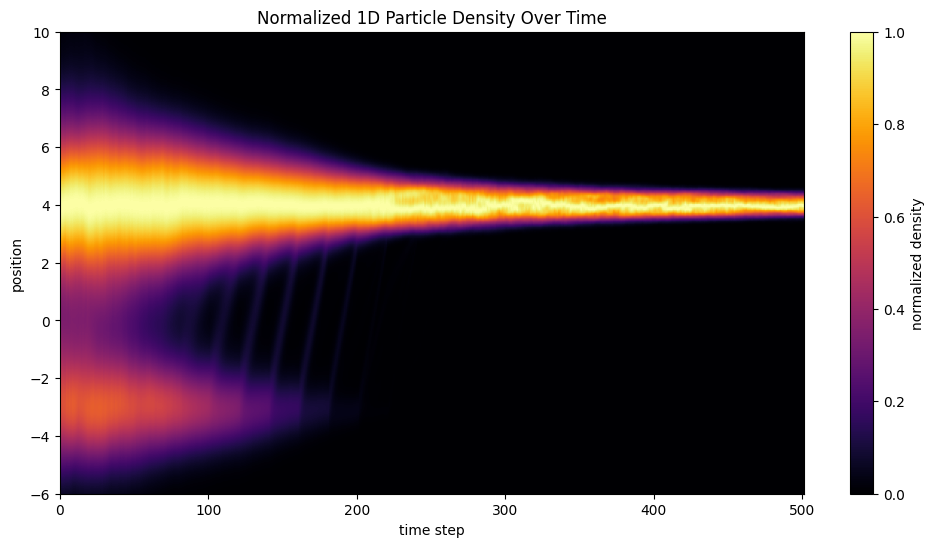}
\includegraphics[width=0.48\textwidth]{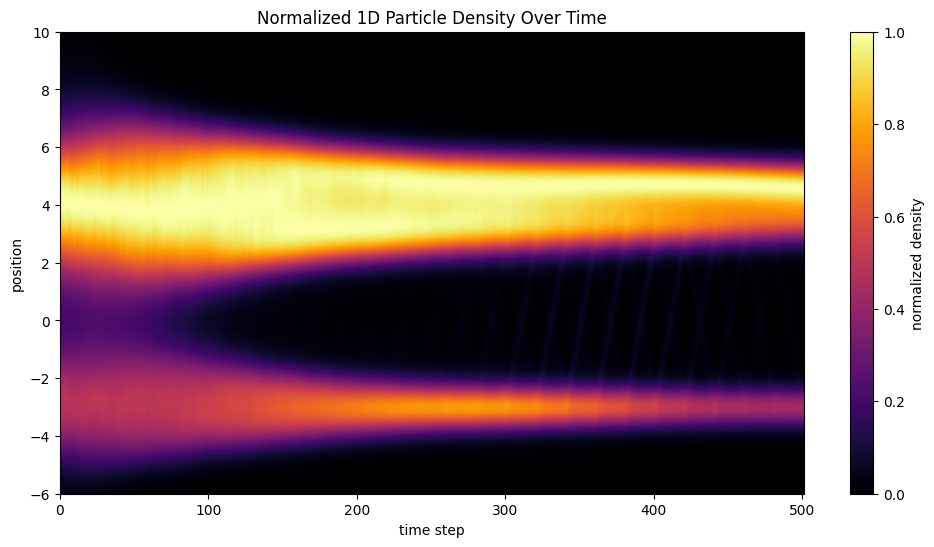}
\includegraphics[width=0.48\textwidth]{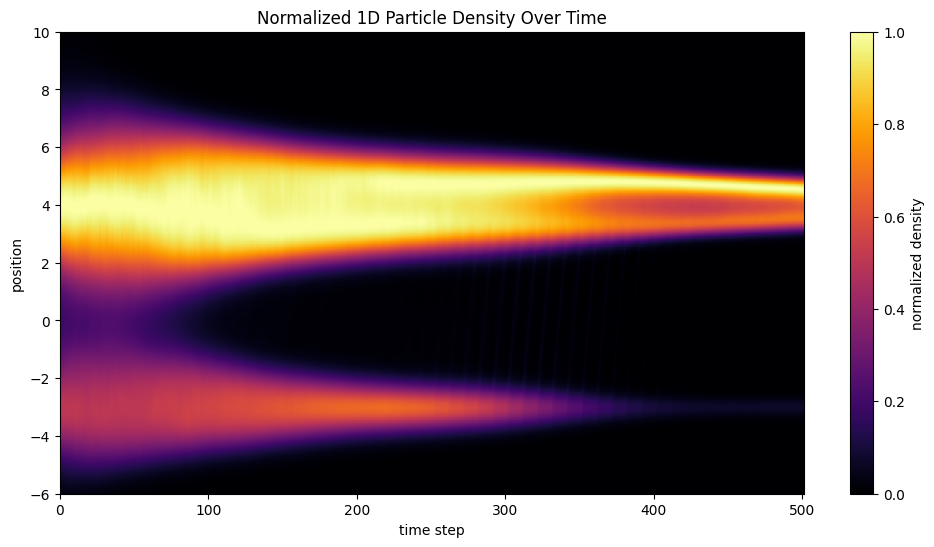}
\caption{Particle heatmaps for CSA (top-left) and Controlled Swarm Gradient (top-right for $m=2$, bottom for $m=6$).  
All runs use 100 experiments with 100 particles each. Time step $\Delta t = 0.002$, velocity estimated every 20 iterations ($h = 0.04$), except for the bottom-right panel where $h = 0.02$.}
\label{fig:heatmaps_doublewell}
\end{figure}

We observe that all methods converge to the global minimum. However, CSA achieves the best overall convergence. One reason for this could be that the controlled swarm is subject to conflicting forces: one guiding the particles toward the minimum and the other encouraging exploration for particles clustered in a local minimum without distinction between local and global minima. Furthermore, the performance of the controlled swarm is more sensitive to the choice of parameters, as illustrated by the case \(m=6\): estimating \(v_t\) every 10 steps improves convergence compared to estimating it every 20 steps. This is probably due to the fact that the coarser estimate \(C^v_{t+h}\) degrades performance for large values of \(m\) or fast cooling schedules (the ratio \(\frac{\beta'(t)} {\beta(t)}\) can explode).

\subsubsection{Controlled Simulated Annealing as the Limit of Controlled Swarm Gradient Dynamics}

An interesting observation from the numerical experiments above is that the behaviour of the controlled swarm scheme appears to converge to that of CSA as \(m \to 1\). Equivalently, one may view controlled swarm dynamics as a variant of CSA with increased variability around local minima, with particles remaining longer trapped when \(m\) is larger. Heuristically, CSA can be interpreted as the limit of the controlled swarm scheme when \(m\) approaches \(1\).

First, the controlled swarm noise (after dividing the function \(\alpha\) by \(2\)) tends to the CSA noise, since for every \(r>0\),
\[
\alpha(r) = 1 + r^{m-1} \;\xrightarrow[m\to1]{}\; 2.
\]
Similarly, the unnormalized marginal density (for a fixed constant \(C\) not depending on \(m\)) becomes proportional to the Gibbs measure \(\mu_t\) in the limit \(m\to1\). As shown in the following statement.

\begin{proposition}
\label{prop:pointwise-limit-rho}
Let $\beta>0$, let $U:\mathbb{R}^d\to\mathbb{R}$ and fix $C\in\mathbb{R}$. For $m>1$ define, for every $x\in\mathbb{R}^d$,
\[
  \rho(x,m,C)
  :=
  \left(
    \frac{1}{m}\,W_0\!\Big(m e^m e^{-2(m-1)\beta\big(U(x)-C\big)}\Big)
  \right)^{\!\frac{1}{m-1}}.
\]
Then for each fixed \(x\in\mathbb{R}^d\),
\[
  \rho(x,m,C) \xrightarrow[m\to1^+]{} e^{-\beta\big(U(x)-C\big)}.
\]
In particular, for fixed \(C\) the pointwise limit is proportional to the Gibbs density \(e^{-\beta U(x)}\).
\end{proposition}

\begin{proof}
Fix \(x\in\mathbb{R}^d\) and \(C\in\mathbb{R}\). Set
\[
  \Delta := \beta\big(U(x)-C\big),
  \qquad m = 1+\varepsilon, \quad \varepsilon>0.
\]
Define
\[
  g(\varepsilon) := (1+\varepsilon)e^{1+\varepsilon}e^{-2\varepsilon\Delta}.
\]
Expanding the exponential factors at first order in \(\varepsilon\) gives
\begin{align*}
  g(\varepsilon)
  &= e(1+\varepsilon)\,e^{\varepsilon}\,e^{-2\varepsilon\Delta}
  = e(1+\varepsilon)\big(1+\varepsilon+o(\varepsilon)\big)\big(1-2\varepsilon\Delta+o(\varepsilon)\big)\\
  &= e(1+2\varepsilon + o(\varepsilon))(1-2\varepsilon\Delta + o(\varepsilon)) \\
  &= e\big[1 + 2\varepsilon(1-\Delta)\big] + o(\varepsilon).
\end{align*}

Recall that \(W_0(e)=1\) and that the derivative of \(W_0\) at \(x=e\) is
\[
  W_0'(e)=\frac{1}{e+e^{W_0(e)}}=\frac{1}{2e}.
\]
Thus a first-order Taylor expansion of \(W_0\) around \(e\) yields
\[
  W_0\big(g(\varepsilon)\big)
  = 1 + \frac{1}{2e}\big(g(\varepsilon)-e\big) + o\big(g(\varepsilon)-e\big)
  = 1 + \frac{1}{2e}\big(e\cdot 2\varepsilon(1-\Delta)\big) + o(\varepsilon)
  = 1 + \varepsilon(1-\Delta) + o(\varepsilon).
\]
Now compute
\[
  \frac{1}{m}W_0\big(g(\varepsilon)\big)
  = \frac{1}{1+\varepsilon}\big(1 + \varepsilon(1-\Delta) + o(\varepsilon)\big).
\]
Expanding \((1+\varepsilon)^{-1}\) and multiplying terms gives
\[
  \frac{1}{m}W_0\big(g(\varepsilon)\big)
  = \big(1 - \varepsilon + o(\varepsilon)\big)\big(1 + \varepsilon(1-\Delta) + o(\varepsilon)\big)
  = 1 - \varepsilon\Delta + o(\varepsilon).
\]
Therefore
\[
  \rho(x,1+\varepsilon,C)
  = \big(1 - \varepsilon\Delta + o(\varepsilon)\big)^{1/\varepsilon}.
\]
Taking logarithms and using the expansion \(\ln(1+z)=z+o(z)\) as \(z\to0\), we obtain
\[
  \ln \rho(x,1+\varepsilon,C)
  = \frac{1}{\varepsilon}\ln\big(1 - \varepsilon\Delta + o(\varepsilon)\big)
  = \frac{1}{\varepsilon}\big(-\varepsilon\Delta + o(\varepsilon)\big)
  = -\Delta + o(1).
\]
Exponentiating yields
\[
  \rho(x,1+\varepsilon,C) = e^{-\Delta + o(1)} = e^{-\beta\big(U(x)-C\big)} e^{o(1)}.
\]
Letting \(\varepsilon\to0^+\) (i.e. \(m\downarrow1\)) gives the claim.
\end{proof}

\bigskip

However, verifying that this heuristic remains valid for the true marginals, when the normalization constant depends on \(m\), is more delicate.

\subsubsection{Performance with Small Number of Particles}

We also tested the algorithm with a small number of particles using the same parameters as in previous experiments. For \(m=6\), only the configuration with a velocity estimate every 10 steps was used. A single sample of \(\rho_0\) (and \(\mu_0\)) with 1000 particles was drawn, and \(C(0)\) was estimated from this sample. Then, 1000 experiments were performed, each selecting 5 particles randomly and executing a CSG (or CSA) on these 5 particles. The median of the minima of the 5 particles over time was analyzed [Figure~\ref{fig:db-minU}]. Both methods work well, with CSA performing slightly better than CSG. For \(m=6\), high variability prevents complete convergence to the global minimum.

\begin{figure}[t]
\centering
\includegraphics[width=0.85\linewidth]{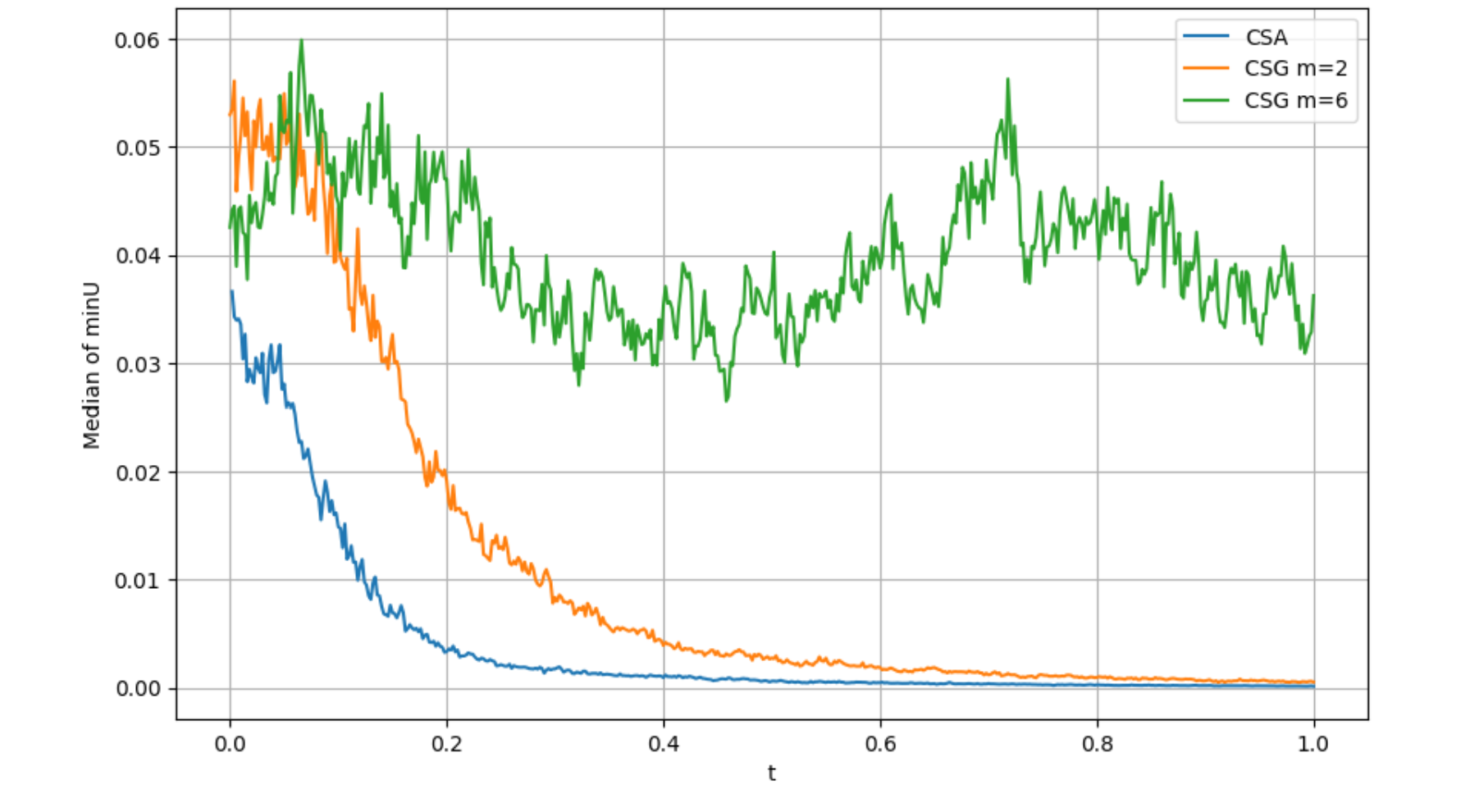}
\caption{
Evolution of the median of the quantity 
\(\displaystyle \min_{1\le i\le 5} \min_t X_t^{(i)}\)
for CSA and CSG with parameters \(m=2\) and \(m=6\).
Each curve represents the median over 1000 experiments using a quadratic cooling schedule :
\(\beta(t)=0.25 + 25\,t^{2}\).
}
\label{fig:db-minU}
\end{figure}

\subsection{Two-Dimensional Six-Hump Camel Function}

We next assessed the robustness of the methods with respect to the initial sampling on the two-dimensional Six-Hump Camel function [Figure~\ref{fig:six_hump}], defined by
\[
U(x_1,x_2)
=
(4 - 2.1 x_1^2 + \tfrac{x_1^4}{3}) x_1^2
+ x_1 x_2
+ (4 x_2^2 - 4) x_2^2
+ 1.0316.
\]
This function admits two symmetric global minima at $(0.0898,-0.7126)$ and $(-0.0898,0.7126)$ for which the value of the function is 0, along with four additional local minima.

\begin{figure}[h!]
\centering
\includegraphics[width=0.7\textwidth]{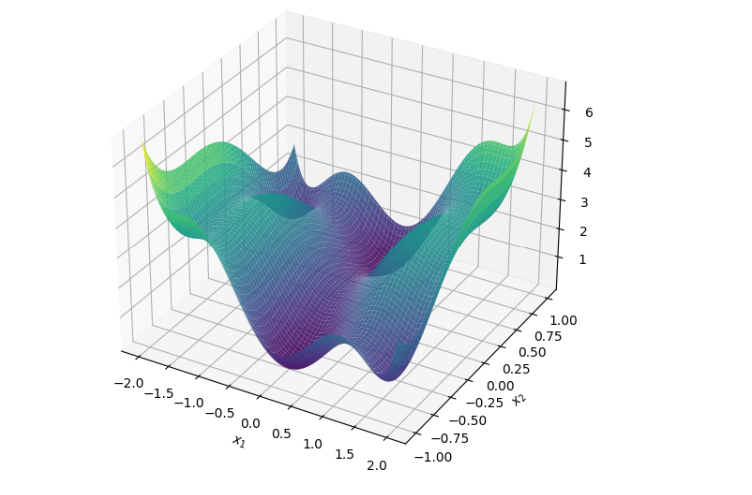}
\caption{2D Six-Hump Camel Function.}
\label{fig:six_hump}
\end{figure}

\medskip

To test sensitivity to the initial configuration, we reproduced the experimental protocol from the last experiment, with the difference that the initial sample used is tightly concentrated around the two local wells.  
A reference sample of $1000$ particles was first drawn from a Gaussian mixture supported near local minima: one Gaussian component was centered near $(2,-1)$ and the second near $(-2,1)$, both with a covariance matrix of \(0.005 I_2\). The two components were sampled with equal probability. 

Then, as in the one-dimensional small-sample experiment, we performed $1000$ independent runs.  
Each run begins by selecting $5$ particles uniformly at random from the reference set, and evolving them under either CSA or CSG dynamics.  
We used a linear cooling schedule
\[
\beta(t) = 0.25 + 25\, t,
\]
and identical discretization parameters as before: time step $\Delta t = 0.002$, and velocity estimation every $h=0.04$ units of time (for CSG with $m=2$ and $m=6$).

\begin{figure}[h!]
\centering
\includegraphics[width=0.8\textwidth]{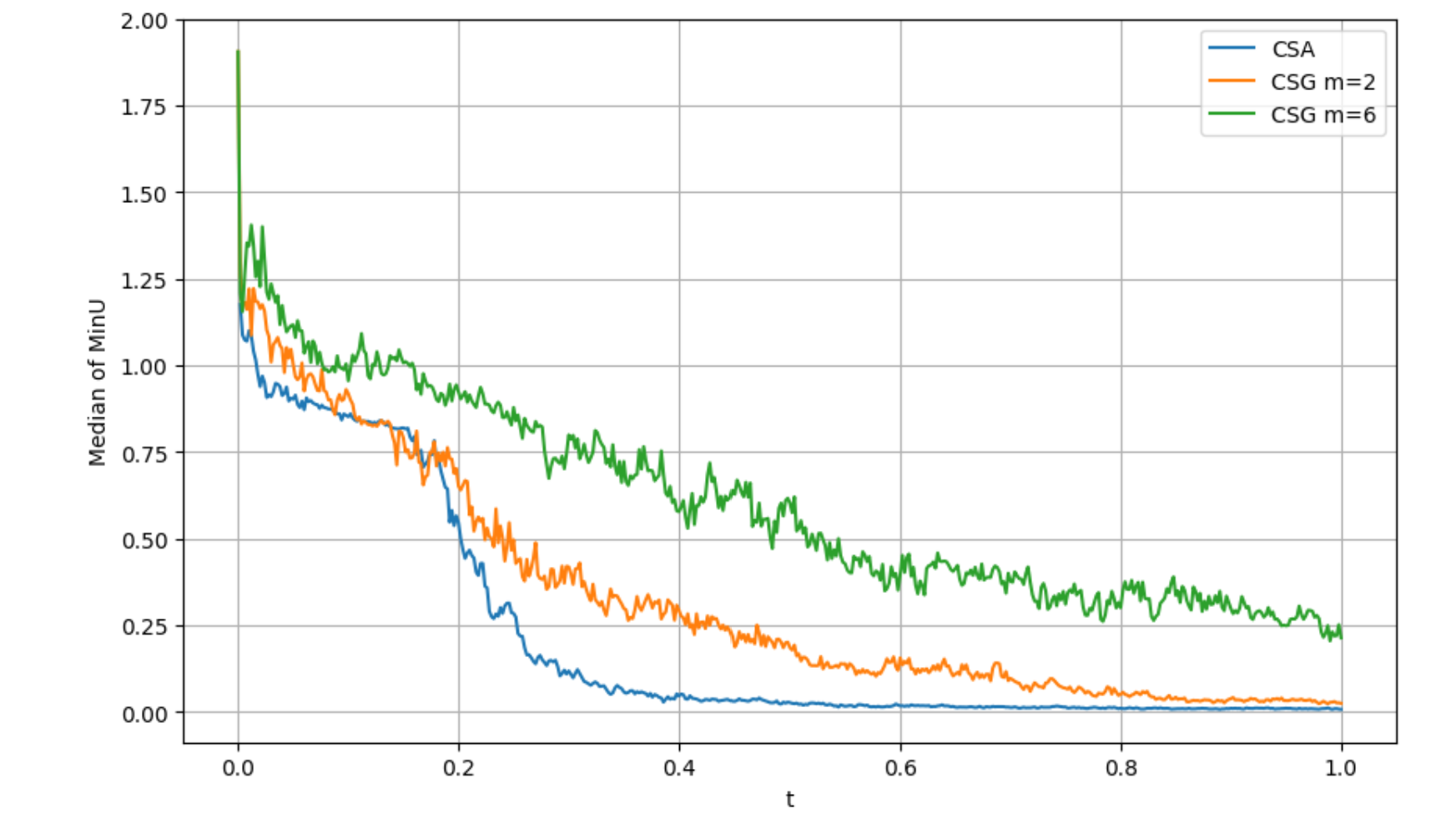}
\caption{
Median of $\min_{1\le i\le 5} U(X_t^i)$ over $1000$ experiments,  
for CSA and CSG ($m=2$ and $m=6$).  
Parameters: $\Delta t = 0.002$, velocity refresh rate $h=0.04$, linear cooling schedule $\beta(t)=0.25+25t$.}
\label{fig:camel-minU}
\end{figure}

\begin{figure}[h!]
\centering
\includegraphics[width=0.8\textwidth]{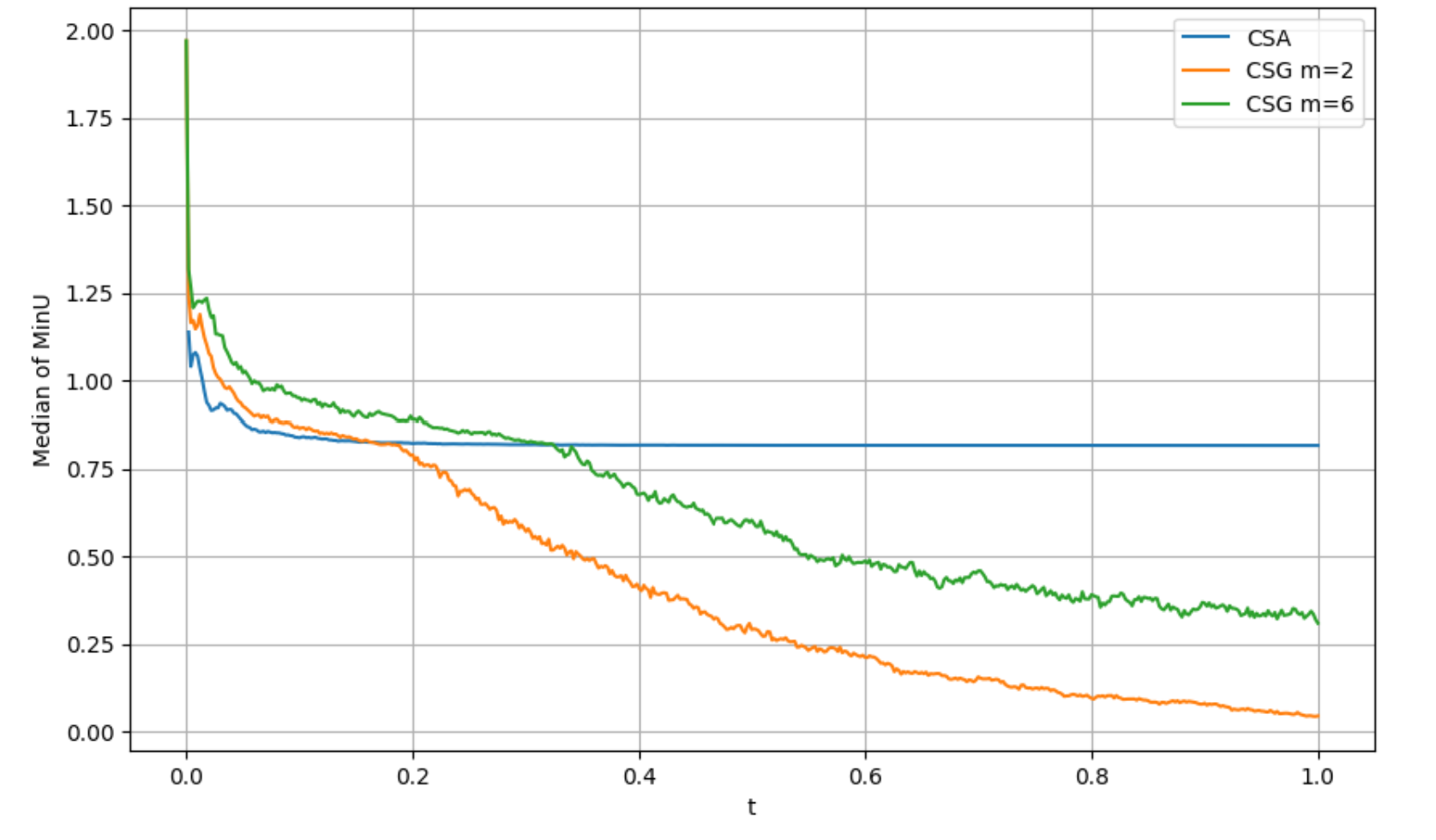}
\caption{
Median of $\min_{1\le i\le 5} U(X_t^i)$ over $1000$ experiments,  
for CSA and CSG ($m=2$ and $m=6$).  
Parameters: $\Delta t = 0.002$, velocity refresh rate $h=0.04$, linear cooling schedule $\beta(t)=0.25+50t$.}
\label{fig:camel-minU2}
\end{figure}

The median of the minima among the 5 selected particles over time was analyzed [Figure~\ref{fig:camel-minU}]. Both methods perform well, particularly CSA and CSG with $m=2$, which manage to reach one of the global minima, with CSA exhibiting slightly better efficiency.  

Interestingly, for this experiment, the CSG demonstrates greater robustness to faster cooling schedules [Figure~\ref{fig:camel-minU2}]. When the cooling rate $\beta(t)$ is doubled, the swarm is still able to escape local minima, whereas CSA fails to do so.

\end{document}